\documentclass[a4paper,12pt]{article}
\usepackage{amsmath,amsthm,amsfonts,amssymb}
\usepackage{graphicx}
\usepackage{cite}
\usepackage{hyperref}
\usepackage[hmargin=1in,vmargin=1in]{geometry}
\hypersetup{colorlinks=true}

\renewcommand{\d}{\mathrm d}
\newcommand{\R}{\mathbb R}
\newcommand{\Z}{\mathbb Z}

\newcommand{\wt}{\widetilde}
\newcommand{\wh}{\widehat}
\newcommand{\ol}{\overline}

\newcommand{\eps}{\varepsilon}

\newcommand{\F}{\mathcal F}
\newcommand{\A}{\mathfrak A}

\renewcommand{\O}{\mathcal{O}}
\renewcommand{\P}{\mathbf P}

\newcommand{\e}{\mathrm e}
\renewcommand{\o}{\mathrm o}
\newcommand{\Be}{\mathrm{RW}}
\newcommand{\Br}{\mathrm{Br}}
\newcommand{\epi}{\mathfrak e}
\newcommand{\LSC}{\mathrm{LSC}}

\newtheorem{proposition}{Proposition}[section]
\newtheorem{theorem}[proposition]{Theorem}
\newtheorem{lemma}[proposition]{Lemma}
\newtheorem{definition}[proposition]{Definition}

\newtheorem{conjecture}[proposition]{Conjecture}
\theoremstyle{definition}

\numberwithin{equation}{section}

\title{The geometry of coalescing random walks, the Brownian web distance and KPZ universality}
\author{B\'alint Vet\H o\thanks{Department of Stochastics, Institute of Mathematics, Budapest University of Technology and Economics
and HUN-REN--BME Stochastics Research Group, M\H uegyetem rkp.\ 3., H-1111 Budapest, Hungary. E-mail: {\tt vetob@math.bme.hu}}
\and
B\'alint Vir\'ag\thanks{Departments of Mathematics and Statistics, University of Toronto, Canada. E-mail: {\tt balint@math.toronto.edu}}}

\begin{document}

\maketitle

\begin{abstract}
Coalescing simple random walks in the plane form an infinite tree.
A natural directed distance on this tree is given by the number of jumps between branches when one is only allowed to move in one direction.
The Brownian web distance is the scale-invariant limit of this directed metric.
It is integer-valued and has scaling exponents $0:1:2$ as compared to $1:2:3$ in the KPZ world.
However, we show that the shear limit of the Brownian web distance is still given by the Airy process.
We conjecture that our limit theorem can be extended to the full directed landscape.
\end{abstract}

\section{Introduction}

\subsection{The discrete web distance}

Consider a system of coalescing random walks $Y$ on the even points
\begin{equation}
\Z^2_\e=\{(i,n)\in\Z^2:i+n\mbox{ is even}\}
\end{equation}
 of the planar lattice
so that
from each $(i,n)\in\Z^2_\e$ there are two outgoing directed edges: to $(i+1,n-1)$ and to $(i+1,n+1)$.
Assign independent random variables $(\xi_{(i,n)})$ with $\P(\xi_{(i,n)}=1)=\P(\xi_{(i,n)}=-1)=1/2$ to the vertices in $\Z^2_\e$.
The \emph{random walk web} $Y$ is a family of coalescing random walks starting at each point of $\Z^2_\e$.
For all $(i,n)\in\Z^2_\e$, set $Y_{(i,n)}(i)=n$ and let
\begin{equation}\label{defY}
Y_{(i,n)}(j+1)=Y_{(i,n)}(j)+\xi_{(j,Y_{(i,n)}(j))}
\qquad\mbox{ for all }j=i,i+1,\dots.
\end{equation}

For any $(i,n;j,m)\in\Z^4$ we define $D^\Be(i,n;j,m)$ to be the smallest integer $k$ such that $(j,m)$ can be reached from $(i,n)$
by following the directed random walk paths in the random walk web $Y$ and by performing $k$ jumps between different random walk paths
along all possible directed paths from $(i,n)$ to $(j,m)$ in the graph $\Z^2_\e$.
The value of $D^\Be(i,n;j,m)$ is infinite if there are no such paths, e.g.\ if $(i,n;j,m)\in\Z^4\setminus\Z^4_\e$.
This defines a directed metric of positive sign in the sense introduced in~\cite{DV21}
hence we call the function $D^\Be$ as \emph{random walk web distance} or \emph{discrete web distance}.

\begin{definition}
For any $(i,n,j,m)\in\Z^4$ let $D^\Be(i,n;j,m)$ be the smallest non-negative integer $k$ such that there are $(i_1,n_1),\dots,(i_k,n_k)\in\Z^2_\e$
with the following property.
There are random walk paths in $Y$ from $(i,n)$ to $(i_1,n_1)$, from $(i_l+1,n_l-\xi_{(i_l,n_l)})$ to $(i_{l+1},n_{l+1})$ for $l=1,\dots,k-1$ and
from $(i_k+1,n_k-\xi_{(i_k,n_k)})$ to $(j,m)$.
We set $D^\Be(i,n;j,m)=\infty$ if there is no such $k$.
\end{definition}

For the function $D^\Be$ it holds that $D^\Be(i,n;i,n)=0$ for all $(i,n)\in\Z^2_\e$
and it satisfies the triangle inequality: for all $(i,n),(k,l),(j,m)\in\Z^2_\e$
\begin{equation}\label{DBtriangle}
D^\Be(i,n;j,m)\le D^\Be(i,n;k,l)+D^\Be(k,l;j,m)
\end{equation}
holds.
Hence $D^\Be$ is a directed metric of positive sign on $\Z^2_\e$ according to Definition 5.1 in~\cite{DV21}.
We mention that $D^\Be(i,n;j,m)$ can alternatively be defined as a first passage time between $(i,n)$ and $(j,m)$
where the passage time of an edge is 0 if it is used by some random walk in $Y$, and 1 otherwise.

The paths of coalescing random walks in $Y$ form a subset of the geodesics of this first passage percolation model and those of the discrete web
distance $D^\Be$.

\subsection{The Brownian web distance}

As the system of coalescing random walks converges to their Brownian counterparts, the discrete web distance also has a limit.
It turns out that the limiting object can be defined based on the Brownian web and its dual, hence we call it the \emph{Brownian web distance}.

Heuristically, the Brownian web contains independent coalescing Brownian motions starting at each point of the space-time $\R^2$.
Based on the work of Arratia~\cite{A81}, the Brownian web was first rigorously constructed by T\'oth and Werner~\cite{TW98}.
In their work, the Brownian web describes the local time profile for the true self-repelling motion.
The term Brownian web first appeared in~\cite{FINR04}.

Let $\Pi_0$ be the space of all continuous functions defined on arbitrary intervals of the type $[t_0,\infty)$.
The Brownian web $B$ is a random closed subset of $\Pi_0$.
It has the property that if $\gamma$ defined on $[t,\infty)$ is in $B$,
then so is $\gamma$ restricted to any interval $[s,\infty)$ for $s\ge t$.
Section \ref{s:BW} contains precise definitions.
It also describes the Hausdorff topology on subsets of a compactification of $\Pi_0$ that gives the Borel $\sigma$-field for the corresponding measure space.

A path $\gamma$ in $B$ defined on $[t,\infty)$ with $\gamma(t)=x$ is called an outgoing path from $(t,x)$.
With probability 1, almost all $(t,x)\in\R^2$ have a unique outgoing path denoted by $B_{(t,x)}$.
Also, almost surely there exist some exceptional points with multiple outgoing paths.
For one type of exceptional points referred to as type $(1,2)$ points with one incoming and two outgoing paths
exactly one of the outgoing paths is the continuation of the incoming path.
This outgoing path of the type $(1,2)$ point $(t,x)$ is denoted by $B_{(t,x)}$ and declared to be the Brownian web path starting from $(t,x)$.
For other types of exceptional points with multiple outgoing paths a unique $B_{(t,x)}$ is not defined.
See Section~\ref{s:BW} for more details.
The Brownian web distance can be introduced as an integer valued function as follows.

\begin{definition}\label{def:DBr}
For any $(t,x,s,y)\in\R^4$ let $D^\Br(t,x;s,y)$ denote the smallest $k\ge 0$
for which there exist points $(t,x)=(t_0,x_0),(t_1,x_1),\dots,(t_{k+1},x_{k+1})=(s,y)\in\R^2$ with $t_0\le t_1<t_2<\dots<t_{k+1}$
and a continuous path $\pi:[t,s]\to\R$ with $\pi(t_i)=x_i$ for $i=0,\dots,k+1$ so that for each $i=0,\dots,k$ there is a path $\gamma_i$ in the Brownian web
with $\pi(r)=\gamma_i(r)$ for all $r\in [t_i,t_{i+1}]$ and the choice of $\gamma_0$ is specified as follows.
It must hold that $\gamma_0=B_{(t,x)}$ if $(t,x)$ is a type $(1,2)$ point
and $t_0=t_1$ must hold if $(t,x)$ is any other exceptional point with multiple outgoing paths.
We set $D^\Br(t,x;s,y)=+\infty$ if there is no such $k\ge 0$.
\end{definition}

We explain some properties of the Brownian web distance that are investigated in this paper.
The scale invariance with exponents $0:1:2$ is an immediate consequence of Brownian scaling.
The value of the exponents is in contrast with the KPZ scaling exponents $1:2:3$, see, for example, Section 2.1 in~\cite{G21}.

\begin{proposition}
For any $\alpha>0$, we have the equality in distribution
\begin{equation}
\left(D^\Br(\alpha^2 t,\alpha x;\alpha^2 s,\alpha y),(t,x;s,y)\in\R^4\right)
\stackrel{\d}{=}\left(D^\Br(t,x;s,y),(t,x;s,y)\in\R^4\right).
\end{equation}
\end{proposition}

A model related to the Brownian web distance is the Brownian castle
which arises as the scaling limit of the infinite temperature version of the ballistic deposition model.
It was introduced in~\cite{CH23} as a scale-invariant Markov process with scaling exponents $1:1:2$
which is different from the one in the Edwards--Wilkinson class as well as from the KPZ class.
The Brownian castle is constructed based on the Brownian web and its dual in the following way.
To each segment of path in the dual Brownian web a centered Gaussian random variable is associated
with variance equal to the length of the corresponding time interval and independently for disjoint paths.
Then the value of the process at a space-time point is the sum of the Gaussian variables along the dual Brownian web path started from the point.

As one may expect, the value of the Brownian web distance $D^\Br(t,x;s,y)$ can change dramatically
by a small perturbation of the starting point $(t,x)$ or by that of the endpoint $(s,y)$.
As we shall see from the description given in Subsection~\ref{ss:BWdistregions} the Brownian web distance is more sensitive to the changes of the endpoint.
Hence the Brownian web distance is not a continuous function of its variables.
However the Brownian web distance has a lower semicontinuous version $D^{\Br,\LSC}(t,x;s,y)$
which only differs from the Brownian web distance $D^\Br(t,x;s,y)$ when $(t,x)$ is an exceptional point in the Brownian web with multiple outgoing paths.
See Definition~\ref{def:DBrLSC} and Proposition~\ref{prop:DBrLSC} for details.

As a consequence, the lower semicontinuous Brownian web distance $D^{\Br,\LSC}$ is a random variable
with values in the space of lower semicontinuous functions on $\R^4$.
Lower semicontinuity of functions is equivalent with having a closed epigraph.
The natural metric is a partially compactified version of the Hausdorff distance between the epigraphs of the functions
where the range of the functions is mapped to a compact interval.
Then the space of lower semicontinuous functions becomes a separable metric space.

For an integer $n$ and $(t,x,s,y)\in \mathbb R^4$ we define the rescaled discrete web distance as
\begin{equation}\label{defDRWn}
D^\Be_n(t,x;s,y)=\begin{cases}
D^\Be(nt,n^{1/2}x;ns,n^{1/2}y),\quad  &\mbox{if } (nt,n^{1/2}x,ns,n^{1/2}y)\in \mathbb Z^4,\\
\infty, & \mbox{otherwise.}
\end{cases}
\end{equation}
We prove the convergence of the rescaled discrete web distance to the Brownian web distance in the epigraph sense which we define below.
Let $f:\R^4\to\ol\R=\R\cup\{\infty\}$ be a lower semicontinuous function.
The epigraph of $f$ is the set
\begin{equation}
\epi f=\{(x,y)\in\R^4\times\ol\R:y\ge f(x)\}
\end{equation}
which is closed by the lower semicontinuity of $f$.
The epigraph of lower semicontinuous functions are elements of the space $\mathcal E_*$,
the set of all closed subsets $\Gamma\subset\R^4\times\ol\R$ such that $\Gamma\cap(\{x\}\times\ol\R)\neq\emptyset$ for all $x\in\R^4$.
We equip $\mathcal E_*$ with the following version of the Hausdorff topology.
Consider the map $E:\R^4\times\ol\R\to\R^4\times[-1,1]$ given by
\begin{equation}
E(r,s)=\left(r,\frac{|u-v|e^{-|r|}s}{1+|s|}\right)
\end{equation}
for any $(r,s)\in\R^4\times\ol\R$ with $r=(u,v)\in\R^4$ with the convention $E(r,\pm\infty)=(r,\pm|u-v|e^{-|r|})$.
For two elements $e_1,e_2\in\mathcal E_*$ we define their distance $d_*(e_1,e_2)$ to be the Hausdorff distance of the images $E(e_1)$ and $E(e_2)$.
The space $(\mathcal E_*,d)$ is compact by Lemma 7.1 of~\cite{DV21}.

\begin{theorem}\label{thm:DBetoDBr}
There is a coupling of the Brownian web $B$ and the sequence $Y_n$ of random walk webs such that
the epigraphs of the lower semicontinuous functions $D^\Be_n$ on $\R^4$ defined in terms of $Y_n$ converge to the epigraph of $D^{\Br,\LSC}$,
that is, $\epi D^\Be_n\to\epi D^{\Br,\LSC}$ in $\mathcal E_*$ almost surely as $n\to\infty$.
\end{theorem}

The proof of Theorem \ref{thm:DBetoDBr} is based on understanding the boundaries of regions that are distance at most $k$ from a spatial half line.
These are closely related to boundaries of disks, see Figure \ref{fig:my_label}.

\begin{figure}
\centering
\includegraphics[width=200pt]{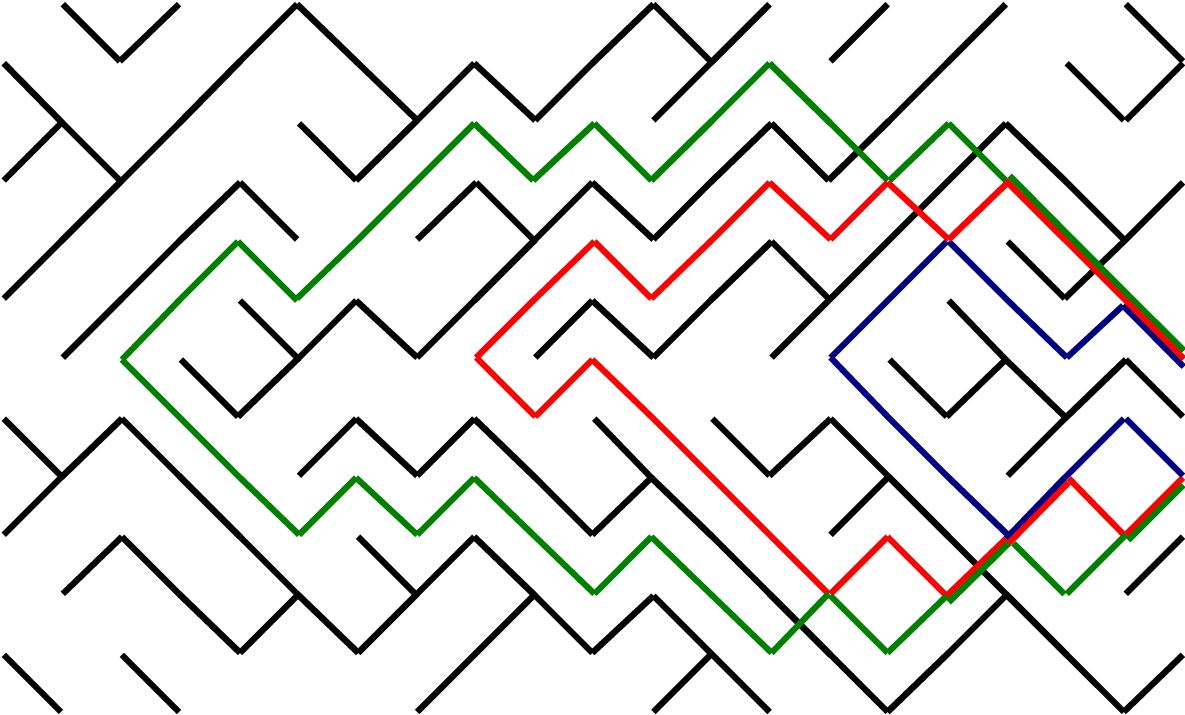}
\qquad\includegraphics[width=200pt]{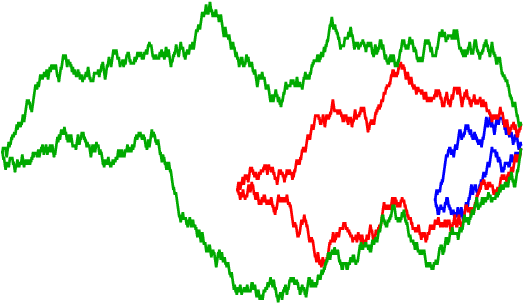}
\caption{Boundaries of  left disks with radii $0$ (in blue), $1$ (in red) and $2$ (in green) for random walk and Brownian web distance.
The centers of nontrivial left Brownian web disks are atypical points.}
\label{fig:my_label}
\end{figure}

We give an alternative characterization of the Brownian web distance as an induced metric given in Definition 5.3 in~\cite{DV21}.
We introduce a natural partial directed metric $d(t,x;s,y)$ defined only if the point $(s,y)$ is on one of the outgoing paths starting at $(t,x)$
and in particular if $s\ge t$.

\begin{definition}\label{def:DBrtilde}
We set $d(t,x;s,y)=0$ if $(t,x)=(s,y)$ or if there is a unique outgoing path $B_{(t,x)}$ starting at $(t,x)$ and $B_{(t,x)}(s)=y$.
We let $d(t,x;s,y)=1$ in the following cases:
either if $(t,x)$ is a type $(1,2)$ point and $(s,y)$ is on the other outgoing path $B'_{(t,x)}$ starting at $(t,x)$
or if $(t,x)$ is another type of exceptional point with multiple outgoing paths and $(s,y)$ is a point along one of these paths.

Let $\wt D^\Br$ be the directed metric induced by $d$ on $\R^4$, that is,
the supremum of all directed metrics on $\R^4$ which are upper bounded by $d$ for those pairs of points for which $d$ is defined.
\end{definition}

We remark that this is the natural way to define $d(t,x;s,y)$ which is $0$ as long as $(s,y)$ can be obtained by following $B_{(t,x)}$.
In case of an exceptional starting point $(t,x)$ the distance is $1$ when following any other path starting at $(t,x)$.

\begin{proposition}\label{prop:DBraltern}
The induced directed metric $\wt D^\Br$ exists and it agrees with $D^\Br$ on $\R^4$.
\end{proposition}

We mention two further remarkable properties of the Brownian web distance
which are formulated in Subsection~\ref{ss:BWdistcont}.
By Proposition~\ref{prop:DBrfiniteness}, the Brownian web distance is finite if and only if
the target point is on the skeleton of the Brownian web.
Furthermore, by Proposition~\ref{prop:DBrrational}, the Brownian web distance is determined by its values on a countable set.

\subsection{KPZ universality}

Even though the Brownian web distance is scale-invariant, in certain directions it still exhibits KPZ universality.
This can be described in terms of the directed landscape.
First, consider independent two-sided standard Brownian motions $W_i$ for $i \in \mathbb Z$, and the directed metric on $ \mathbb R\times \mathbb Z$ given by Brownian last passage percolation
\begin{equation}\label{defBLPP}
L(s,m;t,n)=\max_{s=t_{m-1}\le t_m\le\dots\le t_n=t}\sum_{i=m}^n(W_i(t_i)-W_i(t_{i-1}))
\end{equation} whenever $s\le t, m\le n$, and $L=-\infty$ elsewhere.
It was shown as Theorem 1.5 in~\cite{DOV18} that Brownian last passage percolation has a distributional scaling limit,
which can be taken to be the definition of the directed landscape
\begin{equation}
n^{1/6}\left(
L(s+2xn^{1/3},\lfloor sn \rfloor;t+2yn^{1/3},\lfloor tn \rfloor)
-2(t-s)n^{1/2}-2(y-x)n^{1/6}\right)
\rightarrow \mathcal L(x,s;y,t)
\end{equation}
as a function of $(x,s,y,t)\in \mathbb R^4$, in the topology generated by uniform convergence on compact subsets with $s<t$.  The directed landscape has been shown to be the scaling limit of several last passage percolation models, as well as the KPZ equation, see \cite{DV21,QS23,V20,W23}.

The function $x\mapsto \mathcal L(0,0;x,1)+x^2$ is called the stationary Airy process, first introduced in~\cite{PS02} as the scaling limit of the interface in the polynuclear growth model. For $s<t$ the real random variable  $\mathcal L(x,s;y,t)+(x-y)^2/(t-s)$ has GUE Tracy--Widom law scaled by $(t-s)^{1/3}$.

We prove that the shear limit of the Brownian web distance to a half-line in the first pair of variables is given in terms of the directed landscape.
We use the notation $D^\Br(t,x;s,A)=\inf_{y\in A}D^\Br(t,x;s,y)$ for all $A\subseteq\R$ and the same for $D^\Be$ in what follows.

\begin{theorem}\label{thm:shearlimit}
As $m\to\infty$, the Brownian web distance after a shear mapping satisfies
\begin{equation}\label{shearlimit}
\frac{tm^2+2zm^{4/3}-D^\Br(-t,2tm+2zm^{1/3};0,\R_-)}{m^{2/3}}\to\mathcal L(0,0;z,t)
\end{equation}
in law with respect to the topology of uniform convergence on compact sets for $(z,t)\in\R\times(0,\infty)$.
\end{theorem}

Note that by the $0:1:2$ scale invariance, as a process in $z,t$,
\begin{equation}
D^\Br(-t,2tm+2zm^{1/3};0,\R_-) \;\stackrel{\d}{=} \;D^\Br(-tm^2,2tm^2+2zm^{4/3};0,\R_-).
\end{equation}
Thus, Theorem~\ref{thm:shearlimit} also gives the more customary scaling limit in the direction $(-1,2)$ with the usual scaling exponents.
Note also that all directions $(-1,\eta)$ when $\eta\not=0$ are equivalent by Brownian scaling.

The choice of the target $\{0\}\times\R_-$ in Theorem~\ref{thm:shearlimit} corresponds
to the narrow wedge initial condition for the KPZ fixed point.
It is a natural question to consider other initial conditions, but we do not pursue this direction.

In the $t=1$ case the limit in \eqref{shearlimit} is the parabolic Airy process, $\mathfrak A(y)=\mathcal L(0,0;y,1)$.
We expect that the Airy scaling limit holds in the second space variable as well and the joint limit should be the Airy sheet, $\mathcal S(x,y)=\mathcal L(x,0;y,1)$.
\begin{conjecture}
The rescaled Brownian web distance has Airy fluctuations in the other space variable as well, that is,
\begin{equation}
\frac{m^2-2ym^{4/3}-D^\Br(-1,2m;0,(-\infty,2ym^{1/3}])}{m^{2/3}}\to \mathfrak A(y)
\end{equation}
as $m\to\infty$.
\end{conjecture}

\begin{conjecture}\label{conj:Airysheet}
The fluctuations of the rescaled Brownian web distance are given by the Airy sheet $\mathcal S$, that is
\begin{equation}\label{conjectureAirysheet}
\frac{m^2+2(z-y)m^{4/3}-D^\Br(-1,2m+2zm^{1/3};0,(-\infty,2ym^{1/3}])}{m^{2/3}}\to\mathcal S(y,z)
\end{equation}
as $m\to\infty$.
\end{conjecture}

We expect that if we replace the half-line $(-\infty,2ym^{1/3}]$ by its endpoint $2ym^{1/3}$ in \eqref{conjectureAirysheet}, the conjecture holds in the sense of hypograph convergence.
In our scaling, the half-line corresponds to an increasingly steeper half-wedge,  and so it converges to a narrow wedge initial condition.

Theorem 6.5 in the recent preprint \cite{DZ24} establishes the full convergence of the rescaled Brownian web distance
as well as of the random walk web distance to the directed landscape, in particular, it proves Conjecture~\ref{conj:Airysheet}.

The next results give information about the limiting fluctuations of the random walk web distance.
In any direction different from horizontal, the random walk web distance to a vertical half-line has Airy fluctuations.

\begin{theorem}\label{thm:DBetoTW}
For any $\eta\in(0,1)$,
\begin{multline}
\frac{(1-\eta^2)^{1/6}}{(\eta/2)^{2/3}}n^{-1/3}\Bigg(\frac{1-\sqrt{1-\eta^2}}2n-\frac{\eta^{1/3}(1-\eta^2)^{1/6}}{2^{1/3}}zn^{2/3}\\
-D^\Be(0,0;n,(-\infty,-\eta n-2^{2/3}\eta^{1/3}(1-\eta^2)^{2/3}zn^{2/3}))\Bigg)\to\A(z)
\end{multline}
as $n\to\infty$ uniformly in $z$ on compact intervals.
\end{theorem}

The main term $(1-\sqrt{1-\eta^2})n/2$ in Theorem \ref{thm:DBetoTW} gives a description of the asymptotic shape of disks in $D^\Be$, see Figure \ref{fig:desmos}.
There is no conceptual difficulty in extending Theorem \ref{thm:DBetoTW} to two-parameter convergence as in Theorem \ref{thm:shearlimit}.
However, in \cite{DV21} only hypograph convergence was shown for the Sepp\"al\"ainen--Johansson model.
The required uniform convergence extension is straightforward but too technical for the present paper.

\begin{figure}[ht!]
\centering
\includegraphics[width=50mm]{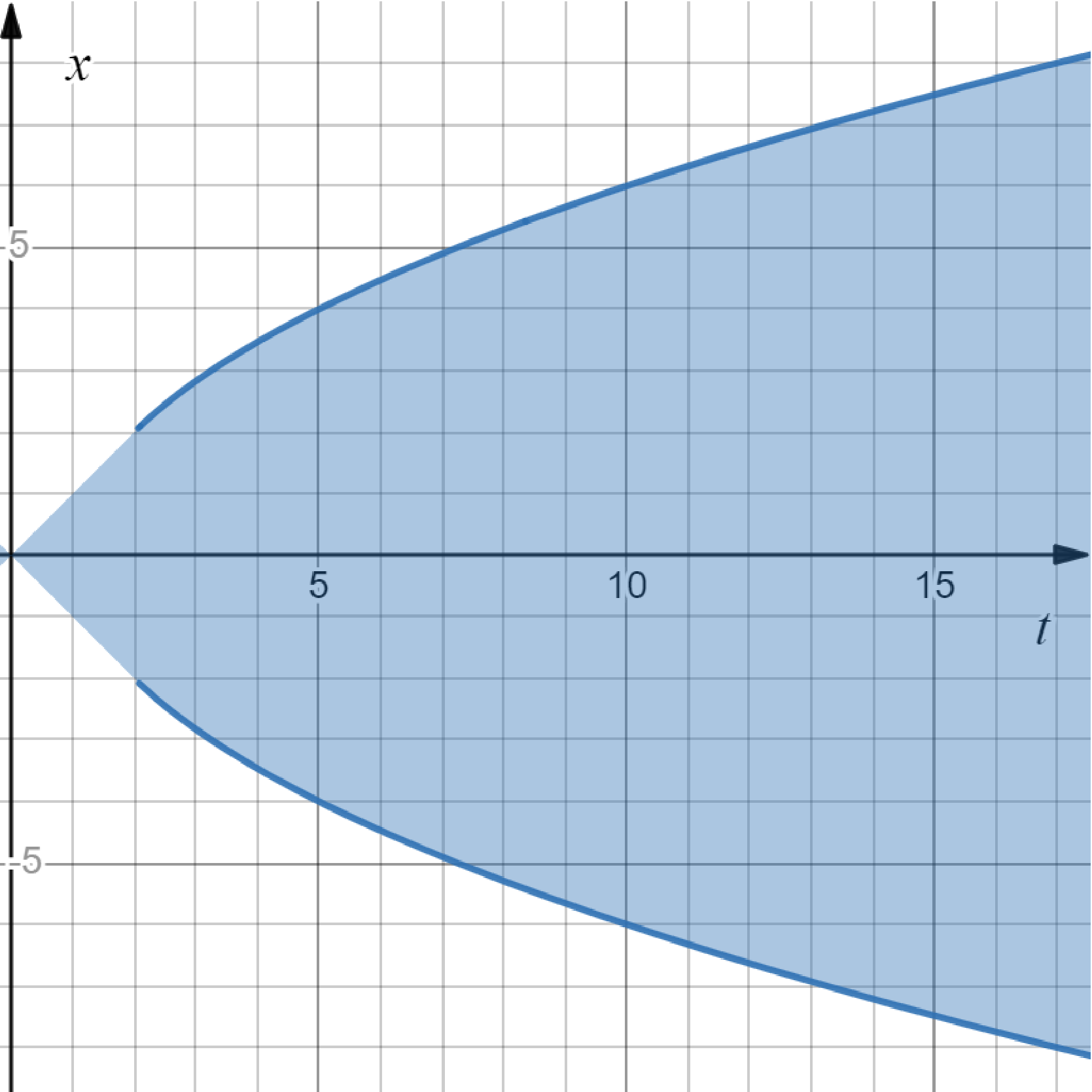}
\caption{The asymptotic disk $t-\sqrt{t^2-x^2}\le 2$ for the shift term in Theorem \ref{thm:DBetoTW}}
\label{fig:desmos}
\end{figure}

The directions different from horizontal are covered by Theorem~\ref{thm:DBetoTW}.
The random walk web distance of two points along the same horizontal line is surprisingly completely different.
We prove logarithmic upper and lower bounds on the horizontal distance and we expect it to satisfy a central limit theorem with normal fluctuations.
The logarithmic scaling of the Brownian web distance in the horizontal direction can be proved similarly.

\begin{theorem}\label{thm:DBehorizontal}
For some $c>1$,
\begin{equation}\label{DBehorizontal}
\lim_{n\to\infty}\P\left(1/c\le\frac{D^\Be(0,0;n,0)}{\log n}\le c\right)=1.
\end{equation}
\end{theorem}

The rest of the paper is organized as follows.
We define the Brownian web in the appropriate metric space with the corresponding topology in Section~\ref{s:BW}
and we describe some properties of the Brownian web in particular the convergence of the discrete web to the Brownian web in Theorem~\ref{thm:YntoB}.
Section~\ref{s:BWdist} contains the most important properties we prove about the Brownian web distance,
which are used for the proof of our main results.
We prove Theorem~\ref{thm:DBetoDBr} about the convergence of discrete web distance to Brownian web distance in Section~\ref{s:convergence}.
Theorem~\ref{thm:shearlimit} about the shear limit of Brownian distance and Theorem~\ref{thm:DBetoTW} on the Tracy--Widom fluctuations of the random walk web distance are proved in Section~\ref{s:shear},
Theorem~\ref{thm:DBehorizontal} about the fluctuations in the horizontal direction is shown in Section~\ref{s:DBehorizontal}.
The proofs of continuity properties of Brownian web distance are postponed to Section~\ref{s:proofs}.

\section{The Brownian web and its dual}
\label{s:BW}

The Brownian web distance is constructed based on the Brownian web and its dual, which we introduce next.
The most natural construction of the Brownian web is based on Theorem~\ref{thm:Brweb} below which can be found as Theorem 2.1 in~\cite{FINR04} explicitly
and it follows from the proof of Theorem 2.1 in~\cite{TW98}.

We first introduce the metric spaces which are used in the definition of the Brownian web according to~\cite{FINR04}.
By writing $\ol\R=\R\cup\{-\infty,\infty\}$, let $\Phi:\ol\R^2\to[-1,1]^2$ be defined as
\begin{equation}
\Phi(t,x)=\left(\tanh(t),\frac{\tanh(x)}{1+|t|}\right)
\end{equation}
and let $\rho$ be a metric which naturally compactifies $\ol\R^2$ given by
\begin{equation}
\rho((t_1,x_1),(t_2,x_2))=\|\Phi(t_1,x_1)-\Phi(t_2,x_2)\|_1.
\end{equation}
Then $(\ol\R^2,\rho)$ is a compact metric space.

For any $t_0\in\ol\R$, let $C[t_0]$ be the set of functions $f:[t_0,\infty]\to\ol\R$ for which $\Phi_2(t,f(t))$ is continuous.
We let
\begin{equation}\label{defPi}
\Pi=\bigcup_{t_0\in\ol\R}\{t_0\}\times C[t_0]
\end{equation}
denote the set of paths along with their starting points.
For a $(t_0,f)\in\Pi$ let $\wh f$ be the extension of $f$ to $\ol\R$ which is equal to the constant $f(t_0)$ on $[-\infty,t_0)$ and agrees with $f$ on $[t_0,\infty]$.
We define the distance $d$ on $\Pi$ by
\begin{equation}
d((t_1,f_1),(t_2,f_2))=\left|\Phi_1(t_1,f_1(t_1))-\Phi_1(t_2,f_2(t_2))\right|\vee\sup_{t\in\ol\R}\left|\Phi_2(t,\wh f_1(t))-\Phi_2(t,\wh f_2(t))\right|.
\end{equation}
With this metric, $(\Pi,d)$ is a complete separable metric space.

Further let $(H,d_H)$ denote the metric space which consists of compact collections of paths in $(\Pi,d)$ with the Hausdorff metric
\begin{equation}\label{defdH}
d_H(K_1,K_2)=\sup_{k_1\in K_1}\inf_{k_2\in K_2}d(k_1,k_2)\vee\sup_{k_2\in K_2}\inf_{k_1\in K_1}d(k_1,k_2).
\end{equation}
The space $(H,d_H)$ is also a complete separable metric space.
Let $\F_H$ be the Borel $\sigma$-algebra generated by the metric $d_H$.

The Brownian web is defined as a random variable taking values in $(H,\F_H)$.

\begin{theorem}\label{thm:Brweb}
The Brownian web $B$ is an $(H,\F_H)$-valued random variable whose distribution is uniquely determined by the properties below.
\begin{enumerate}
\item For any deterministic point $(t,x)\in\R^2$, there is almost surely a unique path $B_{(t,x)}$ starting at $(t,x)$.
\item For any deterministic $n$ and $(t_1,x_1),\dots,(t_n,x_n)\in\R^2$, the joint distribution of $B_{(t_1,x_1)},\dots,B_{(t_n,x_n)}$
is the same as that of coalescing Brownian motions starting at $(t_1,x_1),\dots,(t_n,x_n)$.
\item For any deterministic countable dense set $D$ in $\R^2$, the closure of $\{B_{(t,x)},(t,x)\in D\}$ in $(H,d_H)$ is equal to $B$ almost surely.
\end{enumerate}
\end{theorem}

Then the Brownian web is constructed as follows.
We fix a countable dense subset $D$ of $\R^2$ as in the third property in Theorem~\ref{thm:Brweb}
and we enumerate the points of $D$ as $z_i=(t_i,x_i)$ for $i=1,2,\dots$.
We sample Brownian motions $B_{z_i}$ starting from the points $z_i$ for each $i$ inductively
so that they are independent until they hit one of the trajectories sampled so far and they merge with the trajectory they first hit.
More precisely with an i.i.d.\ sequence of Brownian motions $(\wt B_i(t),t\ge0)_{i=1}^\infty$ and for all $i=1,2,\dots$
we let $B_{z_i}(t)=x_i+\wt B_i(t-t_i)$ for $t\in[t_i,\tau_i)$ where
\begin{equation}\label{deftaui}
\tau_i=\inf\{t\ge t_i:\exists j\in\{1,2,\dots,i-1\}:x_i+\wt B_i(t-t_i)=B_{z_j}(t)\}.
\end{equation}
If $\iota(i)=j\in\{1,2,\dots,i-1\}$ as above in \eqref{deftaui}, that is, for which $x_i+\wt B_i(t-t_i)=B_{z_j}(t)$,
then we define $B_{z_i}(t)=B_{z_{\iota(i)}}(t)$ for $t\ge\tau_i$.
Note that $\tau_1=\infty$ but $\tau_i$ is finite almost surely for $i\ge2$.
The union of the paths $(B_{z_i})_i$ for $z_i\in D$ excluding the starting points $z_i$ is the skeleton of the Brownian web.
Once the skeleton is sampled, the Brownian web $B$ is determined uniquely by Theorem~\ref{thm:Brweb}.
Equivalently, the Brownian web $B$ is the closure of its skeleton in the metric space $(H,d_H)$.

The dual of the Brownian web consists of coalescing Brownian paths running backwards in time and it is called the backward (dual) Brownian web.
It can be constructed based on the same countable dense subset $D$ of $\R^2$ as follows.
The skeleton of the backward Brownian web from each point of $D$ is the almost surely unique continuous curve going backwards in time
which does not cross the forward lines.
The backward Brownian web $\wh B$ is the closure of this set of paths in the metric space of backward paths $(\wh H,d_H)$ with the same $d_H$ as in \eqref{defdH}.
By Theorem 2.3 in~\cite{TW98}, the backward Brownian web has the same distribution as the time-reversed trajectories of the forward web.

Due to Theorem~\ref{thm:Brweb} for any deterministic $(t,x)\in\R^2$, there is almost surely a unique forward path denoted by $B_{(t,x)}$ in $B$ starting at $(t,x)$
and a unique backward path denoted by $\wh B_{(t,x)}$ in $\wh B$.
However, the Brownian web and its dual have random exceptional points with incoming paths or multiple outgoing paths.
They are characterized by their types as follows.

Any two paths $b,b'\in B$ are said to be equivalent paths entering the point $(t,x)$ if $b|_{[t-\varepsilon,t]}=b'|_{[t-\varepsilon,t]}$ for some $\varepsilon>0$.
The number of equivalence classes defines $m_{\rm in}(t,x)$.
Also $m_{\rm out}(t,x)$ can be defined similarly as the number of equivalence classes of outgoing paths.
Then the pair $(m_{\rm in}(t,x),m_{\rm out}(t,x))$ is the type of the point $(t,x)$.
By Proposition 2.4 in~\cite{TW98}, almost surely all points of $\R^2$ have one of the following six types: $(0,1)$, $(0,2)$, $(0,3)$, $(1,1)$, $(1,2)$, $(2,1)$.
For topological reasons, if a point has type $(m_{\rm in},m_{\rm out})$ in the forward Brownian web $B$,
then its type in the backward web $\wh B$ is $(\wh m_{\rm in},\wh m_{\rm out})=(m_{\rm out}-1,m_{\rm in}+1)$.

The $(1,2)$ points have special importance in the construction of the Brownian web distance.
These points can be characterized as follows.
A point is of type $(1,2)$ if and only if both a forward and a backward path pass through it.
The unique incoming path in a $(1,2)$ point $(t,x)$ continues along exactly one of the two outgoing paths which is $B_{(t,x)}$.
Let us introduce the notation $B'_{(t,x)}$ for the other outgoing path starting at the $(1,2)$ point $(t,x)$.
It is also referred to as the newly born path at $(t,x)$ in the literature, see e.g.~\cite{NRS10}.
For other types of exceptional points with multiple outgoing paths a unique $B_{(t,x)}$ is not defined.
Let $B_{(t,x)+}$ and $B_{(t,x)-}$ denote the highest and lowest outgoing paths starting at $(t,x)$
which coincide with $B_{(t,x)}$ for almost all points $(t,x)$.
The backward paths $\wh B_{(t,x)\pm}$ are defined similarly.

Donsker's invariance principle implies the convergence of the random walk web to the Brownian web for any finite collection of paths.
The convergence extends to the full web and its dual as follows.
First we introduce the dual backward system of coalescing random walks $\wh Y$ on the dual lattice
\begin{equation}
\Z^2_\o=\{(i,n)\in\Z^2:i+n\mbox{ is odd}\}.
\end{equation}
For any $(i,n)\in\Z^2_\o$ a path starts at $\wh Y_{(i,n)}(i)=n$ and the paths evolve as
\begin{equation}\label{defYhat}
\wh Y_{(i,n)}(j-1)=\wh Y_{(i,n)}(j)-\xi_{(j-1,\wh Y_{(i,n)}(j))}
\end{equation}
for all $j=i,i-1,\dots$ using the same random variables $\xi_{(i,n)}$ as in \eqref{defY}.
Let $Y^{(n)}$ denote the rescaled random walk web which is defined as follows.
For all $t,x$ with $(nt,n^{1/2}x)\in\Z^2_\e$ we let $Y^{(n)}_{(t,x)}(s)=n^{-1/2}Y_{(nt,n^{1/2}x)}(ns)$ if $ns\in\Z$
and we define $Y^{(n)}_{(t,x)}(s)$ by linear interpolation between these values which yields a continuous function for all $s\ge t$.
Let $\wh Y^{(n)}$ be the rescaled backward random walk web defined similarly.

\begin{theorem}\label{thm:YntoB}
The pair $(B,\wh B)$ of forward and backward Brownian webs together with the rescaled forward and backward random walk webs $(Y^{(n)},\wh Y^{(n)})$
can be realized on the same probability space in a way that
\begin{equation}\label{YntoB}
(Y^{(n)},\wh Y^{(n)})\to(B,\wh B)
\end{equation}
almost surely in $(H,d_H)\times(\wh H,d_H)$.
\end{theorem}

The convergence in distribution in \eqref{YntoB} is the $b=0$ special case of Theorem 5.4 in~\cite{SS08}.
Since $(H,d_H)$ and $(\wh H,d_H)$ are separable metric spaces Skorokhod's representation theorem implies the existence of the coupling
and the almost sure convergence in Theorem~\ref{thm:YntoB}.

\section{Properties of the Brownian web distance}
\label{s:BWdist}

\subsection{Continuity properties of the Brownian web distance}
\label{ss:BWdistcont}

First we give the definition of the lower semicontinuous Brownian web distance below.

\begin{definition}\label{def:DBrLSC}
For any $(t,x,s,y)\in\R^4$ let $D^{\Br,\LSC}(t,x;s,y)$ denote the smallest $k\ge 0$
for which there exist points $(t,x)=(t_0,x_0),(t_1,x_1),\dots,(t_{k+1},x_{k+1})=(s,y)\in\R^2$ with $t_0<t_1<t_2<\dots<t_{k+1}$ and a continuous path $\pi:[t,s]\to\R$
so that for each $i=0,\dots,k$ there is a path $\gamma_i$ in the Brownian web so that $\pi(r)=\gamma_i(r)$ for all $r\in [t_i,t_{i+1}]$.
We set $D^{\Br,\LSC}(t,x;s,y)=+\infty$ if there is no such $k\ge 0$.
\end{definition}

The two versions of the Brownian web distance $D^\Br(t,x;s,y)$ and $D^{\Br,\LSC}(t,x;s,y)$ only differ in the case
when $(t,x)$ is an exceptional point in the Brownian web and the optimal path $\pi$ from $(t,x)$ to $(s,y)$ does not start with $B_{(t,x)}$ or if $B_{(t,x)}$ is not even defined uniquely.
In these cases $D^\Br(t,x;s,y)=D^{\Br,\LSC}(t,x;s,y)+1$.
Otherwise, the two functions are the same.
The function $D^{\Br,\LSC}$ is indeed lower semicontinuous.

\begin{proposition}\label{prop:DBrLSC}
On an event of probability one it holds for all $(t,x,s,y)\in\R^4$ that the mapping $(t,x,s,y)\mapsto D^{\Br,\LSC}(t,x;s,y)$ is lower semicontinuous.
\end{proposition}

The Brownian web distance $D^\Br$ clearly satisfies the triangle inequality but it fails to hold for $D^{\Br,\LSC}$.
For this reason $D^\Br$ is the natural notion of directed metric.
The version $D^{\Br,\LSC}$ has better continuity properties (lower semicontinuous, see Proposition~\ref{prop:DBrLSC})
and it arises as the natural limit of the discrete web distance (see Theorem~\ref{thm:DBetoDBr}).

Next, we show two further remarkable properties of the Brownian web distance.
First, we characterize the finiteness of the Brownian web distance between two space-time points by the fact
that the target is on the skeleton of the Brownian web.
Then we point out that the Brownian web distance on $\R^4$ is determined by its values on a countably infinite subset.
The proof of these statements are postponed to Subsection~\ref{ss:furtherprop}.

\begin{proposition}\label{prop:DBrfiniteness}
Almost surely for all $(t,x,s,y)\in\R^4$ with $t<s$, the Brownian web distance $D^\Br(t,x;s,y)$ is finite if and only if
$(s,y)$ is an interior point of the Brownian web path starting at some $(u,z)$ with $u<s$ and $z\in\R$,
that is, $B_{(u,z)}(s)=y$.
\end{proposition}

\begin{proposition}\label{prop:DBrrational}
The Brownian web distance $D^\Br(t,x;s,y)$ for all $(t,x;s,y)\in\R^4$ is determined by
the values $D^\Br(t,x;s,I)$ where $I=[u,v]$ and $t,x,s,u,v$ are rational.
\end{proposition}

\subsection{Left neighbourhoods of an interval}
\label{ss:BWdistregions}

Proposition~\ref{prop:DBrfiniteness} shows that $D^\Br(t,x;s,y)=\infty$ for any $(t,x,s,y)\in\R^4$
for which $(s,y)$ is not hit by a Brownian path in the Brownian web $B$, that is, for almost all $(s,y)$.
The same applies for the lower semicontinuous Brownian web distance $D^{\Br,\LSC}$ as well.
Hence, we define distances to intervals as follows.
Let $I\subset\R$ be a possibly infinite interval. For $s\in \mathbb R$ let
\begin{equation}\label{defDBrI}
D^{\Br,\LSC}(t,x;s,I)=\inf_{y\in I}D^{\Br,\LSC}(t,x;s,y),\qquad
Q_k(s,I)=\{(t,x):D^{\Br,\LSC}(t,x;s,I)\le k\}.
\end{equation}
These regions are connected subsets of $\R^2$ satisfying $Q_k\subseteq Q_{k+1}$.
Proposition~\ref{prop:regions} below  describes the boundaries of the regions $Q_k$.

We define below the backward paths $\rho_k^\pm$ which according to Proposition~\ref{prop:regions} turn out to be the boundary curves for the regions $Q_k$.
Fix an interval $I=[u,v]$ and let $\rho_0^+(t)=\wh B_{(s,v)}(t)$ and $\rho_0^-(t)=\wh B_{(s,u)}(t)$ for $t\le s$,
that is, the backward Brownian web trajectories starting at time $s$ at the two endpoints of the interval $I=[u,v]$.
Let $\tau_0=\sup\{t\le s:\rho_0^+(t)=\rho_0^-(t)\}$ denote the time when the two backward trajectories meet.
Given the paths $\rho_k^\pm(t)$ for $t\le s$ and their time of collision $\tau_k$, we define
\begin{equation}\label{defrho}
\rho_{k+1}^+(t)=\sup_{r\in[\max(t,\tau_k),s]}\wh B_{(r,\rho_k^+(r))+}(t),\qquad
\rho_{k+1}^-(t)=\inf_{r\in[\max(t,\tau_k),s]}\wh B_{(r,\rho_k^-(r))-}(t)
\end{equation}
for all $t\le s$ and we let $\tau_{k+1}=\sup\{t\le s:\rho_{k+1}^+(t)=\rho_{k+1}^-(t)\}$ be their time of collision.

The next results show that the curves $\rho_k^\pm$ are the boundaries for $Q_k$
and that they arise as Brownian paths reflected off one another in the Skorokhod sense.
The proof of these results are postponed to Subsection~\ref{ss:properties}.

\begin{proposition}\label{prop:Skorokhod}
Let $I=[u,v]\subset\R$ and $s\in\R$ be fixed.
Assume that for all $j=0,1,\dots,k$ the curves $\rho_j^\pm(t)$ are given on $t\in[\tau_j,s]$.
Conditionally given these curves the distributions of $\rho_{k+1}^\pm(t)$
are reflected backward Brownian paths off $\rho_k^\pm(t)$ in the Skorokhod sense on $t\in[\tau_k,s]$
and independent Brownian motions until collision at $\tau_{k+1}$ on $t\in[\tau_{k+1},\tau_k]$.
In particular, $\rho_{k+1}^\pm(t)$ are continuous.
\end{proposition}

\begin{proposition}\label{prop:regions}
Let $I=[u,v]\subset\R$ and $s\in\R$ be fixed.
The union of the interval $s\times I$ and the curves $\rho_k^\pm(t)$ on $t\in[\tau_k,s]$ is the boundary of the region $Q_k$ given in \eqref{defDBrI}.
\end{proposition}

\subsection{Left neighbourhoods in the discrete web}

Similarly to \eqref{defDBrI}, we introduce
\begin{equation}\label{defRk}
D^\Be(i,n;j,I)=\inf_{m\in I}D^\Be(i,n;j,m),\qquad R_k=\{(i,n):D^\Be(i,n;j,I)\le k\}
\end{equation}
for the discrete web distance for any interval $I\subset\Z$.
$R_k$ is the discrete analogue of the region $Q_k$ in \eqref{defDBrI} which depends on the choice of $j$ and $I$.
The regions are connected subsets of $\Z^2_e$ for which $R_k\subseteq R_{k+1}$ clearly holds.
Next we define the backward paths $r_k^\pm$ which according to Proposition~\ref{prop:discreteregions} turn out to serve as boundary curves for the regions $R_k$.

Let $j\in\Z$ and $I=[u,v]\subset\Z$ be fixed such that $(j,u),(j,v)\in\Z^2_\e$.
We let $r_0^+(i)=\wh Y_{(j,v+1)}(i)$ and $r_0^-(i)=\wh Y_{(j,u-1)}(i)$ for $i=j,j-1,\dots$
which are the backward discrete Brownian web trajectories starting from the points $(j,v+1)$ and $(j,u-1)$ on the dual lattice $\Z^2_\o$ at the endpoints of $j\times I$.
Then we let $T_0=\max\{i\le j:r_0^+(i)=r_0^-(i)\}$ be the time when the two backward trajectories meet.

Next we define the two sequences of paths $r_k^\pm$ starting at $(j,v+1)$ and $(j,u-1)$ by induction on $k$ as follows.
We use the convention $r_k^+(j+1)=v$ and $r_k^-(j+1)=u$ for all $k$.
Given $r_k^\pm(i)$ for $i\in[T_k,j]$ we define
\begin{equation}\label{defr}
r_{k+1}^+(i)=\max_{l\in\{i,\dots,j\}}\wh Y_{(l,r_k^+(l+1)+1)}(i),\qquad
r_{k+1}^-(i)=\min_{l\in\{i,\dots,j\}}\wh Y_{(l,r_k^-(l+1)-1)}(i)
\end{equation}
for all $i\le j$ and we let $T_{k+1}=\max\{i\le T_k:r_{k+1}^+(i)=r_{k+1}^-(i)\}$.
Note that $r_{k+1}^+(j)=v+1$ and $r_{k+1}^-(j)=u-1$ holds automatically by \eqref{defr}.
The definition means that $r_{k+1}^+$ is the maximum of all backward discrete Brownian web paths
starting at the points of $r_k^+$ and $r_{k+1}^-$ is the minimum of the paths starting along $r_k^-$.

The next results are the discrete analogues of Propositions~\ref{prop:Skorokhod} and \ref{prop:regions}.
The statements are proved in Subsection~\ref{ss:discreteproperties}.

\begin{proposition}\label{prop:discreteSkorokhod}
Conditionally given the curves $r_j^\pm$ for all $j=0,1,\dots,k$,
the path $r_{k+1}^+$ evolves as the discrete upward Skorokhod reflection of a backward random walk off $r_k^+$ until time $T_k$.
More precisely there exists a backward random walk $s_{k+1}$ such that
\begin{equation}\label{rkSkorokhod}
r_{k+1}^+(i)=s_{k+1}(i)-\min_{l\in\{i,\dots,j\}}\left(s_{k+1}(l)-r_k^+(l+1)-1\right)
\end{equation}
holds for all $i\in[T_k,j]$.
Furthermore, there are backward random walks $s_0,s_1,\dots,s_k$ such that $r_k^+$ can be represented as a discrete last passage percolation as
\begin{equation}\label{rkwithLPP}
r_k^+(i)=\max_{i=l_k\le l_{k-1}\le\dots\le l_{-1}=j-k}\sum_{m=0}^k\left(s_m(l_m)-s_m(l_{m-1})\right)+k+v+1.
\end{equation}
Similarly for $r_{k+1}^-$ which evolves as the discrete downward Skorokhod reflection of a backward random walk off $r_k^-$ until $T_k$ conditionally given the curves $r_k^\pm$.
\end{proposition}

\begin{proposition}\label{prop:discreteregions}
The paths $r_k^\pm(i)$ for $i=T_k,T_k+1,\dots,j$ together with the interval $j\times I$ serve as the boundary of the region $R_k$ for all $k=0,1,2,\dots$.
\end{proposition}

\section{Convergence of discrete web distance to Brownian web distance}
\label{s:convergence}

This section is devoted to the proof of Theorem~\ref{thm:DBetoDBr}, the main result about the epigraph convergence of $D^\Be_n$ to $D^{\Br,\LSC}$.
The convergence is proved to hold in $(\mathcal E_*,d_*)$ which is a compact metric space by Lemma 7.1 of~\cite{DV21}.
We first state the lower semicontinuous analogue of Lemma 7.3 of~\cite{DV21}
which provides a sufficient condition for the convergence in $(\mathcal E_*,d_*)$ to hold.
We postpone its proof to the end of the section.

\begin{lemma}\label{lemma:epigraphconv}
Let $f,f_n:\R^4\to\ol\R$ be lower semicontinuous functions.
Assume that for any convergent sequence $x_n\to x\in\R^4$ it holds that
\begin{equation}\label{epigraphcond1}
\liminf_{n\to\infty}f_n(x_n)\ge f(x).
\end{equation}
Further assume that for any $x\in\R^4$ we can find a convergent sequence $x_n\to x$ such that
\begin{equation}\label{epigraphcond2}
\limsup_{n\to\infty}f_n(x_n)\le f(x).
\end{equation}
holds.
Then the epigraphs converge, that is, $\epi f_n\to\epi f$ in $\mathcal E_*$ as $n\to\infty$.
\end{lemma}

The strategy of the proof of Theorem~\ref{thm:DBetoDBr} is
to understand the level curves of the distance functions $D^\Be_n$ and $D^{\Br,\LSC}$.
One could hope that the convergence of these curves is a deterministic consequence of the random walk web converging to the Brownian web.
We do not have a proof of this.
The main difficulty is that the random walk distance uses microscopic information (two curves being a single edge away from each other)
that does not behave well in the scaling limit (which only sees two curves being less-than-scaling away from each other).
The remedy we use is Skorokhod reflection.

We fix an $s\in\R$ and an interval $I=[u,v]$.
For all $n$, we define the regions $R_k=R_k(ns,n^{1/2}I)$ given in \eqref{defRk} with target time $ns$ and target interval $n^{1/2}I=[n^{1/2}u,n^{1/2}v]$.
The boundary curves of these regions are the curves $r_k^\pm$ with starting points $(ns,n^{1/2}v)$ and $(ns,n^{1/2}u)$ by Proposition~\ref{prop:discreteregions}.
Let the rescaled boundary curves be denoted by
\begin{equation}\label{defrkn}
r_{k,n}^\pm(t)=n^{-1/2}r_k^\pm(nt)
\end{equation}
for all $t\le s$ and with linear interpolation between integer values.
In the next proposition, we establish the uniform convergence of the rescaled curves $r_{k,n}^\pm$
to $\rho_k^\pm$ on compact intervals in probability as $n\to\infty$ under the coupling given in Theorem~\ref{thm:YntoB}.
The proof is based on Donsker's invariance principle and by the continuity of Skorokhod reflection.
Furthermore, the convergence of single curves can be improved to countable families.

\begin{proposition}\label{prop:regionconv}
Consider a countable collection of $s_i,I_i$ where $s_i\in\R$ and $I_i\subset\R$ is a closed interval for all $i$.
Then for all $i$ the rescaled boundary curves $r_{k,n}^\pm$ of the regions $R_k(ns_i,n^{1/2}I_i)$ defined in \eqref{defrkn}
converge in probability to the boundaries $\rho_k^\pm$ of $Q_k(s_i,I_i)$ uniformly on compact sets.
\end{proposition}

\begin{proof}[Proof of Proposition~\ref{prop:regionconv}]
We fix an $s\in\R$ and an $I=[u,v]\subset\R$ first and we consider $Q_k$ defined in \eqref{defDBrI}.
By Proposition~\ref{prop:regions}, $\rho_k^\pm(r)$ for $r\le s$ are boundary curves of $Q_k$.
We study the convergence to these curves.
The definition \eqref{defr} of $r_k^+$ can be written in terms of the rescaled boundary as
\begin{equation}
r_{k,n}^+(t)=\sup_{r\in[t,s]}\wh Y^{(n)}_{(r,r_{k-1,n}^+(r+1/n)+1/\sqrt n)}(t).
\end{equation}
for $t\le s$.

We fix the natural coupling of the random walk web and the Brownian web given in Theorem~\ref{thm:YntoB}.
In particular, the rescaled backward random walk web $\wh Y^{(n)}$ converges to $\wh B$
as compact collection of paths in the Hausdorff topology almost surely.
Let $\Gamma_n$ be the closure of the set of all paths in $\wh Y^{(n)}$ started from the points
$(r,r_{k-1,n}^+(r+1/n)+1/\sqrt n)$ for $r\in[t,s]$.
Since $\Gamma_n$ is a closed subset of the compact set $\wh Y^{(n)}$ and by the convergence $\wh Y^{(n)}\to\wh B$
in $(H,d_H)$, the sets $\Gamma_n$ have a limit $\Gamma$ that is a subset of $\wh B$.

We claim that $\Gamma$ is a subset of $\Gamma'$, the closure of the  set of all paths in $\wh B$ started from $(r,\rho_{k-1}^+(r))$.
Indeed, let $\gamma_n\in\Gamma_n$ be a convergent sequence.
Then the set of starting points converges and all starting points in the limit lie on $\rho_{k-1}^+$, since $r_{k-1,n}^+\to\rho_{k-1}^+$.

This implies that
\begin{equation}
\bigcup_{\gamma \in \Gamma_n} \operatorname{graph}(\gamma)
\to\bigcup_{\gamma \in \Gamma} \operatorname{graph}(\gamma)
\subset \bigcup_{\gamma \in \Gamma'} \operatorname{graph}(\gamma)
\end{equation}
in the Hausdorff topology, and by considering the upper boundaries of these sets, we get
\begin{equation}
\limsup_{n\to\infty}r_{k,n}^+(r)\le\rho_k^+(r)
\end{equation}
for all $r$ almost surely.
But note that $r_{k,n}^+\to\rho_k^+$ in law with respect to uniform convergence on compact sets
(by the Skorokhod reflection representation and Donsker's theorem).
These two statements imply $r_{k,n}^+\to\rho_k^+$ in probability uniformly on compact sets.
The argument works simultaneously for countably many pairs $(s,I)$ and for finitely many $k$ hence the proof is complete.
\end{proof}

Note that the proof above works only for countable collections of pairs $(s,I)$ and not for all $(s,I)$ because of the convergence in law part.

The statement of Proposition~\ref{prop:regionconv} can be realized by a coupling as almost sure convergence as follows.
Under the coupling given in Theorem~\ref{thm:YntoB} the region boundaries $r_{k,n}^\pm$ converge uniformly in probability to $\rho_k^\pm$.
Since the region boundaries are continuous, their reversed versions take values in the space $\Pi$ defined in \eqref{defPi}.
We have that the random walk web and its dual along with countably many region boundaries converge jointly in probability.
The product of $H\times\wh H$ with a countable number of copies of $\Pi$ is still separable.
Hence Skorokhod's representation theorem guarantees the existence of a coupling under which
the forward and backward random walk webs and countably many region boundaries converge almost surely.
We choose the points $s_i$ as being rational and $I_i$ as being the intervals with rational endpoints and we fix the coupling described above.

\begin{proof}[Proof of Theorem~\ref{thm:DBetoDBr}]
We check the two conditions of Lemma~\ref{lemma:epigraphconv}.
To see \eqref{epigraphcond1} for any sequence $x_n\to x$ in $\R^4$ we use the coupling described above.
This coupling ensures that as $(H,F_H)$-valued random variables the rescaled discrete web $Y^{(n)}$
converges to the Brownian web $B$ almost surely in $(H,d_H)$ as $n\to\infty$.

If the left-hand side of \eqref{epigraphcond1} is equal to $k$ for $D^\Be_n$, then there is a sequence of convergent starting points $(t^{(n)},x^{(n)})\to(t,x)$
and convergent endpoints $(s^{(n)},y^{(n)})\to(s,y)$ such that $D^\Be_n(t^{(n)},x^{(n)};s^{(n)},y^{(n)})=k$.
This means that for all $n$ there are $k+1$ paths $\pi_0^{(n)},\dots,\pi_k^{(n)}$ in $Y^{(n)}$
along the geodesic between $(t^{(n)},x^{(n)})$ and $(s^{(n)},y^{(n)})$ with $k$ jumps.
The convergence of $Y^{(n)}$ to $B$ in $(H,d_H)$ implies that for all $j=0,1,\dots,k$ and for all $n$ there are paths $\wt\pi_j^{(n)}$ in $B$
so that $d(\pi_j^{(n)},\wt\pi_j^{(n)})\to0$ for all $j=0,\dots,k$ by \eqref{defdH}.
Since $B$ as an element of $H$ is a compact collection of paths in $(\Pi,d)$,
the sequence $\wt\pi_j^{(n)}$ must have a subsequential limit $\wt\pi_j^{(\infty)}$ for all $j=0,\dots,k$.
One can find the limiting paths $\wt\pi_j^{(\infty)}$ so that they realize a geodesic for $D^\Br$ between $(t,x)$ and $(s,y)$.
This shows that $D^{\Br,\LSC}(t,x;s,y)\le k$ proving \eqref{epigraphcond1}.

Next, we use the coupling described after Proposition~\ref{prop:regionconv}
where the region boundaries converge for rationals $s$ and intervals $I$ with rational endpoints.
To prove the second condition \eqref{epigraphcond2} in Lemma~\ref{lemma:epigraphconv}, assume that $k:=D^{\Br,\LSC}(t,x;s,y)<\infty$.
Then we can find intervals $I_j=[u_j,v_j]$ with rational endpoints and rationals $s_j$ so that
\begin{itemize}
\item the line segment $\{s_j\}\times I_j$ intersects the geodesic from $(t,x)$ to $(s,y)$,
\item $|s_j-s|,|u_j-y|,|v_j-y|\le 1/j$.
\end{itemize}
Then $D^{\Br,\LSC}(t,x;s_j,I_j)\le k$.
To prove \eqref{epigraphcond2}, we construct a sequence of indices $n_j$ such that there are $t^{(n_j)},x^{(n_j)},s^{(n_j)},y^{(n_j)}\in\R$ with
$(t^{(n_j)},x^{(n_j)},s^{(n_j)},y^{(n_j)})\to(t,x,s,y)$ so that
\begin{equation}
\limsup_{j\to\infty}D^\Be_{n_j}(t^{(n_j)},x^{(n_j)};s^{(n_j)},y^{(n_j)})\le k.
\end{equation}
We define $n_j$ inductively. We set $n_0=1$, and given $n_{j-1}$ we let $n_j\ge n_{j-1}$ so that
the $k$th rescaled boundaries $r_{k,n_j}^\pm=n_j^{-1/2}r_k^\pm(n_j\cdot)$ corresponding to $s_j,I_j$ are at most $1/j$ away from their limits $\rho_k^\pm$ uniformly on the interval $[t-1,s_j]$.
This can be achieved by the coupling based on Proposition~\ref{prop:regionconv}.

Next, we choose a starting point $(t^{(n_j)},x^{(n_j)})$ which is at most $1/j$ away from $(t,x)$ so that $D^\Be_{n_j}(t^{(n_j)},x^{(n_j)};s_j,y^{(n_j)})\le k$
for some $y^{(n_j)}\in[u_j,v_j]$.
Then $(s^{(n_j)},y^{(n_j)})=(s_j,y^{(n_j)})$ is also at most $2/j$ far from $(s,y)$ showing the existence of the subsequence $n_j$ with the required properties.
This proves \eqref{epigraphcond2} and the theorem.
\end{proof}

\begin{proof}[Proof of Lemma~\ref{lemma:epigraphconv}]
The space $\mathcal E_*$ consists of closed subsets of $\R^4\times\ol\R$.
By the lower semicontinuity of $f_n$ the epigraphs $\epi f_n$ are closed.
Let $\Gamma$ denote a subsequential limit of $\epi f_n$ in $\mathcal E_*$. This exists by the compactness of $\mathcal E_*$.
Define the function $g:\R^4\to\ol\R$ by
\begin{equation}
g(x)=\inf\{z\in\ol\R:(x,z)\in\Gamma\}.
\end{equation}
Since $\Gamma$ is a closed subset of $\R^4\times\ol\R$ as an element of $\mathcal E_*$,
the function $g$ is lower semicontinuous and $\epi g=\Gamma$.
Next, we show that $g=f$.

On one hand, $\Gamma$ is a subsequential limit of $\epi f_n$ in $\mathcal E_*$.
Hence for any $x\in\R^4$, the point $(x,g(x))\in\Gamma$ can be approximated by points in $\epi f_n$, that is, there is a sequence $x_n\to x$ such that
\begin{equation}
\liminf_{n\to\infty}f_n(x_n)=g(x).
\end{equation}
Comparing this with \eqref{epigraphcond1}, which holds for any convergent sequence $x_n$ yields that $g\ge f$.

On the other hand, for any $x\in\R^4$ there is a convergent sequence $x_n$
for which $\lim_{n\to\infty}f_n(x_n)=f(x)$ holds by \eqref{epigraphcond2} and \eqref{epigraphcond1}.
Hence the convergence
\begin{equation}\label{Exnfnconv}
E(x_n,f_n(x_n))\to E(x,f(x))
\end{equation}
is also satisfied.
The convergence of a subsequence of the epigraphs $\epi f_n$ to $\Gamma$ in $\mathcal E_*$ means that
a subsequence of $E(\epi f_n)$ converges to $E(\Gamma)$ in the Hausdorff distance.
This fact together with the convergence \eqref{Exnfnconv} implies for the limit on the right-hand side of \eqref{Exnfnconv}
that $E(x,f(x))\in(E(\Gamma))_{+\varepsilon}$ for all $\varepsilon>0$.
Since $E(\Gamma)$ is a closed set, this means that $E(x,f(x))\in E(\Gamma)$ and $(x,f(x))\in\Gamma$, which yields $f\ge g$.
\end{proof}

\section{The KPZ limit}
\label{s:shear}

This section contains the proofs of Theorem~\ref{thm:shearlimit} and Theorem~\ref{thm:DBetoTW}.
Both proofs are based on Propositions~\ref{prop:regions} and \ref{prop:discreteSkorokhod} describing the boundary curves of the regions with different distances.  These curves are Brownian motions and random walks reflected off each other in the Skorokhod sense.
The distribution of the reflected paths can be represented as last passage values.

In the case of Theorem~\ref{thm:shearlimit} about the Brownian web distance we specialize Proposition~\ref{prop:regions} to the semiinfinite interval $I=\R_-=(-\infty,0]$.
According to the proposition the upper boundary of the region with distance at most $k$
\begin{equation}\label{rhokshear}
\rho_k^+(t)=\sup\{x\in\R:D^{\Br,\LSC}(t,x;0,\R_-)\le k\}=\sup\{x\in\R:D^\Br(t,x;0,\R_-)\le k\}
\end{equation}
for $t\le0$ is a backward Brownian motion reflected off $\rho_{k-1}^+$ in the Skorokhod sense.
The initial $\rho_0^+$ is a backward Brownian motion.
By Lemma~\ref{lemma:BLPPSkorokhod} below the processes $\rho_k^+(-t)$ for $k=0,1,2,\dots$ have the distribution of Brownian last passage percolation.

We introduce the Brownian last passage percolation with general boundary condition below, see also \eqref{defBLPP}.
Let $f:\R_+\to\R$ be a continuous function and for a non-negative integer $n$ and for $t\ge0$ let
the Brownian last passage percolation with boundary condition $f$ be given by
\begin{equation}\label{defBLPPwithboundary}
L^f(t,n)=\sup_{0\le t_1}\left(f(t_1)+L(t_1,2;t,n)\right)
=\sup_{0=t_0\le t_1\le\dots\le t_n=t}\left(f(t_1)+\sum_{i=2}^n(W_i(t_i)-W_i(t_{i-1}))\right)
\end{equation}
where $W_2(t),W_3(t),\dots,W_n(t)$ are independent standard Brownian motions.

\begin{lemma}\label{lemma:BLPPSkorokhod}
Let $f:\R_+\to\R$ be a fixed continuous function.
For the Brownian last passage percolation with boundary condition $f$ it holds that $L^f(t,1)=f(t)$ and for $n=2,3,\dots$, $L^f(t,n)$ can be represented as a Skorokhod reflection as
\begin{equation}\label{BLPPrecursion}
L^f(t,n)=W_n(t)-\inf_{s\in[0,t]}(W_n(s)-L^f(s,n-1)).
\end{equation}
\end{lemma}

\begin{proof}
For $n=1$, \eqref{BLPPrecursion} holds by definition.
By taking out the supremum over $t_n\in[0,t]$ in \eqref{defBLPPwithboundary} one gets
\begin{equation}
L^f(t,n)=\sup_{t_n\in[0,t]}\left(L^f(t_n,n-1)+W_n(t)-W_n(t_n)\right).
\end{equation}
This is the recursion \eqref{BLPPrecursion}.
\end{proof}

The second building block in the proof of Theorem~\ref{thm:shearlimit} is the following convergence result of Brownian last passage percolation
to the directed landscape from~\cite{DOV18,DV21}, e.g.\ Theorem 1.7 in~\cite{DV21}, see also \cite{FN11,F10,S15} for single-parameter versions.

\begin{proposition}\label{prop:rhoAiry}
The convergence of the rescaled processes
\begin{equation}\label{rhoAiry}
n^{1/6}\left(\rho_{tn+2zn^{2/3}}^+(-t)-2t\sqrt n-2zn^{1/6}\right)\to \mathcal L(0,0;z,t)
\end{equation}
holds as $n\to\infty$ in distribution uniformly in $z$ and $t>0$ in any compact region.
\end{proposition}

\begin{proof}[Proof of Proposition~\ref{prop:rhoAiry}]
With the notations of Theorem 1.7 and Remark 1.10 in~\cite{DV21},
we have $n=\sigma^3$ and $v=(1,-1),e=(0,-1)$.
On the other hand since $\|(x,-y)\|_d=2\sqrt{xy}$ for the Brownian last passage percolation, the Taylor expansion
\begin{equation}
\|v+xe\|_d=\|(1,-1-x)\|_d=2\sqrt{1+x}=2+x-\frac{x^2}4+\O(x^3)
\end{equation}
holds as $x\to0$.
Remark 1.9 in~\cite{DV21} implies that $\alpha=2,\beta=1,\chi/\tau^2=1/4$ and $\chi=1$ for the Brownian last passage percolation.
Then $\tau=2$ holds as a consequence and $v=(0,-2)$, so we can choose $s=0,x=0,y=z$ and leave $t>0$ a parameter.
By (1.6) in Remark 1.10 of~\cite{DV21} with $\wt\alpha=2,\wt\eta=-1$ we have $\wt\gamma=2,\wt\chi=\wt\omega=1$.
Note that Brownian last passage percolation in~\cite{DV21} is obtained by optimizing over down-right paths
whereas our definition \eqref{defBLPP} is the supremum over up-right paths.
Hence the Brownian last passage time between $(0,0)$ and $(tn,-tn-2zn^{2/3})$ in~\cite{DV21} translates into
\begin{equation}\label{BLPPasymp}
n^{-1/3}\left(L(0,0;tn,tn+2zn^{2/3})-2tn-2zn^{2/3}\right)\to\mathcal L(0,0;z,t).
\end{equation}
The distributional identity $L(0,0;ns,k)\stackrel{\d}=n^{1/2}L(0,0;s,k)$ for the Brownian last passage percolation follows from Brownian scaling.
Using it in \eqref{BLPPasymp} and writing the convergence in terms of the boundaries $\rho_k^+$ gives \eqref{rhoAiry}.
\end{proof}

\begin{proof}[Proof of Theorem~\ref{thm:shearlimit}]
First, we fix a compact interval $K$ for $z$.
The convergence in distribution in \eqref{rhoAiry} happens in the space of continuous functions.
We consider the supremum distance of continuous functions on $K$ which generates the topology of uniform convergence on $K$.
By Skorokhod's representation theorem, the convergence in distribution in \eqref{rhoAiry} on $K$
can be realized by a coupling of the sequence $\rho_k^+$ for $k=0,1,2,\dots$ and the limit as an almost sure convergence which is uniform on $K$.

The almost sure convergence which is uniform on $K$ in \eqref{rhoAiry} means that for almost all realizations of the randomness it holds that
for every $\varepsilon>0$ there exists a random $k_0$ such that for all $k\ge k_0$
\begin{equation}\begin{aligned}
2t\sqrt n+2zn^{1/6}+(\mathcal L(0,0;z,t)-\varepsilon)n^{-1/6}
&\le\rho_{tn+2zn^{2/3}}^+(-t)\\
&\le2t\sqrt n+2zn^{1/6}+(\mathcal L(0,0;z,t)+\varepsilon)n^{-1/6}
\end{aligned}\end{equation}
holds.
These bounds on the boundary imply that the distance $D^\Br$ with $\varepsilon$ replaced by $2\varepsilon$ satisfies
\begin{equation}\label{DBrupperlower}\begin{aligned}
D^\Br(-t,2t\sqrt n+2zn^{1/6}+(\mathcal L(0,0;z,t)+2\varepsilon)n^{-1/6};0,\R_-)&\ge tn+2zn^{2/3},\\
D^\Br(-t,2t\sqrt n+2zn^{1/6}+(\mathcal L(0,0;z,t)-2\varepsilon)n^{-1/6};0,\R_-)&\le tn+2zn^{2/3}.
\end{aligned}\end{equation}
Now let $m(n)=\sqrt n+\frac L{2t}n^{-1/6}$ where $L\in\R$ is a parameter.
Taylor expansion yields
\begin{equation}\label{LkTaylor}\begin{aligned}
2tm(n)+2z(m(n))^{1/3}&=2t\sqrt n+2zn^{1/6}+Ln^{-1/6}+\O(n^{-1/2})\\
t(m(n))^2+2z(m(n))^{4/3}-L(m(n))^{2/3}&=tn+2zn^{2/3}+\O(1)
\end{aligned}\end{equation}
as $n\to\infty$.
Applying \eqref{LkTaylor} in \eqref{DBrupperlower} with $L=\mathcal L(0,0;z,t)\pm2\varepsilon$ and using the fact that the inequalities in \eqref{DBrupperlower}
hold for any $\varepsilon>0$ if $n$ is large enough implies that
\begin{equation}
D^\Br(-t,2tm+2zm^{1/3};0,\R_-)=tm^2+2zm^{4/3}-\mathcal L(0,0;z,t)m^{2/3}+\O(1)
\end{equation}
as $m\to\infty$ which proves \eqref{shearlimit}.
\end{proof}

The strategy of the proof of Theorem~\ref{thm:DBetoTW} is to rewrite the boundary curve $r^+_k$ in terms of the last passage value in the Sepp\"al\"ainen--Johansson model, Proposition~\ref{prop:discreteSkorokhod}.
Then we use the known fluctuation results of the Sepp\"al\"ainen--Johansson model to conclude Airy fluctuations of the discrete web distance.

Let $T(m,n)$ denote the last passage time from $(0,0)$ to $(m,n)$ in the Sepp\"al\"ainen--Johansson model with parameter $1/2$
where all vertical edges of $\Z^2$ have $0$ passage time and the horizontal edges have weight $0$ or $1$ with probability $1/2$ each
and all edge weights are independent random variables.
Our notation differs from \cite{Joh01} but the first passage time from $(0,0)$ to $(m,n)$ is equal in distribution to $m-T(m,n)$ where $T(m,n)$ is the last passage time.
By \cite{Joh01} for a single $z$ and by Corollary 6.11 of \cite{DNV23} for all $z$ uniformly on compact intervals it holds that if $x>y$ then
\begin{multline}\label{FPPTW}
T\left(xn+2x\frac{(x-y)^{1/3}}{(xy)^{1/3}}zn^{2/3},yn\right)\\
=\left(x-\frac{(\sqrt x-\sqrt y)^2}2\right)n
+(x+\sqrt{xy})\frac{(x-y)^{1/3}}{(xy)^{1/3}}zn^{2/3}
+\frac{(x-y)^{2/3}}{2(xy)^{1/6}}\A(z)n^{1/3}+o(n^{1/3})
\end{multline}
as $n\to\infty$.
By the same idea including Skorokhod's representation as in the proof of Theorem~\ref{thm:shearlimit} we have that for any $\varepsilon>0$ the left-hand side of \eqref{FPPTW} can be almost surely upper and lower bounded by the right-hand side of \eqref{FPPTW} with $\A(z)$ replaced by $\A(z)\pm\varepsilon$ uniformly in $z$ on a compact interval.

\begin{proof}[Proof of Theorem~\ref{thm:DBetoTW}]
Proposition~\ref{prop:discreteSkorokhod} gives a representation of the boundary curve $r_k^+$ as a last passage time which can be written in terms of $T(m,n)$ as
\begin{equation}\label{LPPFPP}\begin{aligned}
r_k^+(i)&=\max_{i=l_k\le l_{k-1}\le\dots\le l_{-1}=j-k}\sum_{m=0}^k\left(s_m(l_m)-s_m(l_{m-1})\right)+k+v+1\\
&\stackrel{\d}=2T(j-k-i,k)+v+2k+i-j+1
\end{aligned}\end{equation}
where the second identity above maps the $\pm1$ weights of horizontal edges into $0$ or $1$ weights in the last passage percolation.

We specify the target interval to be $\{0\}\times(-\infty,0]$, hence we choose $j=0$ and $v=0$.
Further, we set $i=-n$ and $k=\kappa n+czn^{2/3}+\beta n^{1/3}=\kappa\wt n$ with $\kappa$, $c$ and $\beta$ to be determined and with $\wt n=n+\frac{cz}\kappa n^{2/3}+\frac\beta\kappa n^{1/3}$.
We get from \eqref{FPPTW} and \eqref{LPPFPP} that
\begin{equation}\label{rkappabeta}\begin{aligned}
r_k^+(-n)&=2T\left((1-\kappa)n-czn^{2/3}-\beta n^{1/3},\kappa n+czn^{2/3}+\beta n^{1/3}\right)\\
&\qquad-n+2(\kappa n+czn^{2/3}+\beta n^{1/3})+1\\
&=2T\left((1-\kappa)\wt n-\frac c\kappa\wt z\wt n^{2/3},\kappa\wt n\right)-n+2\kappa\wt n
\end{aligned}\end{equation}
where $\wt z=z+(-\frac{2c}{3\kappa}z^2+\frac\beta c)n^{-1/3}+o(n^{-1/3})$.
Then the right-hand side of \eqref{rkappabeta} is of the form found in \eqref{FPPTW} with $x=1-\kappa$, $y=\kappa$, $n=\wt n$, $z=\wt z$ and $c=-2\kappa x(x-y)^{1/3}/(xy)^{1/3}$.
Then \eqref{FPPTW} yields that
\begin{multline}\label{rkappaexpansion}
r_k^+(-n)=2\sqrt{\kappa(1-\kappa)}n+2(1-2\kappa)^{4/3}(\kappa(1-\kappa))^{1/6}zn^{2/3}\\
+\left(\frac{1-2\kappa}{\sqrt{\kappa(1-\kappa)}}\beta+\frac{(1-2\kappa)^{2/3}}{(\kappa(1-\kappa))^{1/6}}\A(z)+o(1)\right)n^{1/3}
\end{multline}
where the $o(1)$ term above is almost surely uniform as $z$ varies in a compact interval.
We choose $\kappa=(1-\sqrt{1-\eta^2})/2$ so that \eqref{rkappaexpansion} simplifies to
\begin{equation}\label{retabeta}
r_k^+(-n)
=\eta n+2^{2/3}\eta^{1/3}(1-\eta^2)^{2/3}zn^{2/3}+\frac{2\sqrt{1-\eta^2}}\eta\left(\beta+\frac{\eta^{2/3}}{2^{2/3}(1-\eta^2)^{1/6}}\A(z)+o(1)\right)n^{1/3}.
\end{equation}
This means that with
\begin{equation}\label{betavalue}
\beta=-\frac{\eta^{2/3}}{2^{2/3}(1-\eta^2)^{1/6}}\A(z)-\varepsilon
\end{equation}
for any $\varepsilon>0$ fixed the right-hand side of \eqref{retabeta} is at most $\eta n+2^{2/3}\eta^{1/3}(1-\eta^2)^{2/3}zn^{2/3}$ if $n$ is large enough.
This implies that
\begin{multline}\label{DBelower}
D^\Be\left(-n,\eta n+2^{2/3}\eta^{1/3}(1-\eta^2)^{2/3}zn^{2/3};0,\R_-\right)\\
\ge\frac{1-\sqrt{1-\eta^2}}2n-\frac{\eta^{1/3}(1-\eta^2)^{1/6}}{2^{1/3}}zn^{2/3}-\frac{\eta^{2/3}}{2^{2/3}(1-\eta^2)^{1/6}}\A(z)-\varepsilon
\end{multline}
for any $\varepsilon>0$ if $n$ is large enough.
By changing the sign of $\varepsilon$ to positive in \eqref{betavalue} we see that the right-hand side of \eqref{retabeta} is at least $\eta n+2^{2/3}\eta^{1/3}(1-\eta^2)^{2/3}zn^{2/3}$, which yields the corresponding upper bound on the random walk web distance in \eqref{DBelower} with $-\varepsilon$ replaced by $+\varepsilon$.
This completes the proof since $D^\Be(-n,\eta n+2^{2/3}\eta^{1/3}(1-\eta^2)^{2/3}zn^{2/3};0,\R_-)$ and $D^\Be(0,0;n,(-\infty,-\eta n-2^{2/3}\eta^{1/3}(1-\eta^2)^{2/3}zn^{2/3}))$ have the same distribution by the shift invariance.
\end{proof}

\section{Discrete web distance in the horizontal direction}
\label{s:DBehorizontal}

This section is devoted to the proof of Theorem~\ref{thm:DBehorizontal} about the horizontal scaling of the random walk web distance.

\begin{proof}[Proof of Theorem~\ref{thm:DBehorizontal}]
First we prove the lower bound in \eqref{DBehorizontal}.
Since $D^\Be(0,0;n,0)$ and $D^\Be(-n,0;0,0)$ have the same distribution
we can consider the upper and lower boundary curves $r_k^\pm$ of the regions $R_k$ corresponding to the single point $(0,0)$ in the place of the target interval $\{j\}\times I$.
The upper and lower curves $r_k^\pm$ meet at time $T_k$ for each $k$.
We will find a dominating sequence $E_k \ge |T_k|$ so that the increments of $\log E_k$ have a uniform exponential tail decay.

The sequence $E_k$ is chosen so that the region $R_k$ is included in the rectangle $[-E_k,0]\times[-\sqrt{E_k},\sqrt{E_k}]$.
This is done by induction on $k$ as follows.
Assume that $E_k$ along with $R_k$ is given.
Then we sample two backward random walks $e_{k+1}^\pm$ starting at $e_{k+1}^\pm(0)=\pm\sqrt{E_k}$ which are reflected off the constant levels $\pm\sqrt{E_k}$ upwards respectively downwards
on the time interval $[-E_k,0]$ and the two random walks run beyond $-E_k$ until they collide.
We let the paths $e_{k+1}^\pm$ follow the steps of the backward random walk web $\wh Y$ when they are away from the barriers at $\pm\sqrt{E_k}$.

Let  $E_{k+1}$ be the maximum of the absolute value of the collision time and the square of the largest absolute value that these two random walks ever had until collision.
In other words, $E_{k+1}$ is the smallest number for which the full trajectories of the two random walks $e_{k+1}^\pm$
are included in the rectangle $[-E_{k+1},0]\times[-\sqrt{E_{k+1}},\sqrt{E_{k+1}}]$.

Let $\mathcal F_k=\sigma(E_1,\ldots E_k)$. We claim that the tail of the conditional law of $\log E_{k+1}-\log E_k$ given $\mathcal F_k$ decays exponentially uniformly in $k$.
Thus the increments are dominated by an i.i.d.\ sequence with finite mean, which then implies the lower bound of the theorem.

First we study the time $-S$ when the two processes $e^\pm_{k+1}$ meet. Until time $-E_k$ these are random walks which are reflected from the boundaries of the rectangle.
Afterwards, their difference performs a lazy random walk. For lazy walks started at $s>0$, the time $\tau$ they hit zero satisfies $\P(\tau>r)\le 3s/\sqrt{r}$, see Corollary 2.28 in~\cite{LPW06}.
So for our walks
\begin{equation}
\P\big(S>(e^x+1)E_k\,|\,\mathcal F_k, e_{k+1}^\pm(-E_k)\,\big)\le \frac{e_{k+1}^+(-E_k)-e_{k+1}^-(-E_k)}{e^{x/2}\sqrt{E_k}}
\end{equation}
if $x$ is large.
Taking conditional expectations with respect to $\mathcal F_k$ we get that
\begin{equation}
\P\big(S>(e^x+1)E_k\,|\,\mathcal F_k\,\big)\le \frac{c\sqrt{E_k}}{e^{x/2}\sqrt{E_k}}=ce^{-x/2}.
\end{equation}
The walks $e^+_{k+1}$ and $-e^-_{k+1}$ are dominated by random walks reflected off the level $\sqrt{E_k}$ upwards all the way until time $-S$ which have the same distribution as the absolute value of random walks shifted by $\sqrt{E_k}$.
This helps us to bound the tail of $M=\max_{-S\le j\le 0} \max(e^+_{k+1}(j), -e^-_{k+1}(j))$. Let $X_j$ be a simple random walk on $\mathbb Z$ independent of the rest.
Then
\begin{equation}\begin{aligned}
\P\big(M>(e^x+1)\sqrt{E_k},\, S\le(e^x+1)E_k\,|\,\mathcal F_k\,\big)
&\le2\P\left(\max_{ j\le(e^x+1)E_k} |X_j|\ge e^{x}\sqrt{E_k}\,|\,\mathcal F_k\right)\\
&\le8\P\left( X_{(e^x+1)E_k}\ge e^{x}\sqrt{E_k}\,|\,\mathcal F_k\right)\\
&\le ce^{-x/2}
\end{aligned}\end{equation}
for large $x$ where Chebyshev's inequality is used in the last line.
So we get
\begin{multline}
\P\left(E_{k+1}>e^{2x}E_k\,|\,\mathcal F_k\right)\\
\le\P\big(S>(e^x+1)E_k\,|\,\mathcal F_k\big)+\P\left(M>(e^x+1)\sqrt{E_k}, S\le (e^x+1)E_k\,|\,\mathcal F_k\right)\le ce^{-x/2}.
\end{multline}
This implies the lower bound, since
\begin{equation}
\P\left(D^\Be(-n,0;0,0)<a\log n\right)\le\P\left(|T_{a\log n}|\ge n\right)
\le\P\left(\log E_{a\log n}\ge \log n\right)\to 0
\end{equation}
by the law of large numbers, as long as $a>0$ is small enough.

For the upper bound it suffices to show  that there is an $\varepsilon>0$ so that $|T_{k+1}|/|T_k|>1+\varepsilon$ with a uniformly positive probability as $k$ varies.
Since, $r_{k+1}^\pm$ are reflected walks,  there are independent simple random walks $s_{k+1}^\pm(n)$ on $n=0,1,\dots,|T_k|$ so that the $r_{k+1}^+(-n)$ stochastically dominates $s_{k+1}^+(n)$ and $s_{k+1}^-(n)$ stochastically dominates $r_{k+1}^-(-n)$ for all $n=0,1,\dots,|T_k|$.
The rescaled distance $(r_{k+1}^+(T_k)-r_{k+1}^-(T_k))/\sqrt{|T_k|}$ therefore stochastically dominates $(s_{k+1}^+(|T_k|)-s_{k+1}^-(|T_k|))/\sqrt{|T_k|}$. The latter  converges to a normal random variable as $|T_k|\to\infty$. This holds since $|T_k|\ge2k$.

This implies that there is a $\varepsilon>0$ for which it takes at least $\varepsilon|T_k|$ steps beyond $T_k$ for $r_{k+1}^\pm$ to collide with probability $p>0$ for all $k$. So  $\log |T_k|$ dominates $\log(1+\eps)$ times a Bernoulli random walk with success probability $p$.
For any
$0<\nu<p\log(1+\eps)$ by the law of large numbers,
\begin{equation}
\P\left(|T_{(\log n)/\nu}|\le n\right)\to0
\end{equation}
as $n\to\infty$.
Assume now that $|T_{(\log n)/\nu}|\ge n$ and the point $(-n,0)$ is not in the set $R_{(\log n)/\nu}$
but $R_
{(\log n)/\nu}$ intersects the vertical line $\{-n\}\times\Z$ above $0$.
As a consequence of Proposition~\ref{prop:regionconv} the sequence $D^\Be(-n,0;0,\Z_+)$ converges in distribution to $D^{\Br,\LSC}(-1,0;0,\R_+)$
and in particular $D^\Be(-n,0;0,\Z_+)$ is tight.
Hence there is a path with a tight number of steps in the forward random walk web from $(-n,0)$ to $\{0\}\times\Z_+$ which must cross $R_
{(\log n)/\nu}$.
It proves that $D^\Be(-n,0;0,0)$ is at most a constant times $\log n$ with probability tending to $1$ as required for the upper bound in \eqref{DBehorizontal}.
\end{proof}

\section{Proofs of continuity properties}
\label{s:proofs}

\subsection{Properties of continuous region boundaries}
\label{ss:properties}

In this subsection, we prove Propositions \ref{prop:Skorokhod} and \ref{prop:regions}.
First, we give in alternative characterization of the boundary curves $\rho_k^\pm$ in Proposition~\ref{prop:rhochar} below which is used in the proofs of Propositions~\ref{prop:Skorokhod} and \ref{prop:regions}.
Proposition~\ref{prop:Skorokhod} is shown using two different couplings of the same sequence of processes
that approximate the Brownian motion reflected off a function in the Skorokhod sense.
The proof of Proposition~\ref{prop:regions} uses the continuity of the paths $\rho_k^\pm$ which follows from Proposition~\ref{prop:Skorokhod}.

Next we give an alternative description of the boundary curves $\rho_k^\pm$ which we introduce at $\wt\rho_k^\pm$ which are shown to be the same in Proposition~\ref{prop:rhochar}.
Let $\wt\rho_0^\pm(t)$ and $\wt\tau_0$ to be the same as $\rho_0^\pm(t)$ and $\tau_0$, that is, the backward Brownian web paths $\wh B_{(s,v)}(t)$ and $\wh B_{(s,u)}(t)$ and $\wt\tau_0$ their collision time.
Then we proceed by induction on $k$.
We assume that the paths $\wt\rho_k^\pm(t)$ and their collision time $\wt\tau_k$ are given and we introduce $\wt\rho_{k+1}^\pm(t)$ as follows.
We first define a decreasing sequence of trajectories $\wt\rho_{k+1}^{+,n}(t)$ which converge to $\wt\rho_{k+1}^+(t)$.
We start the trajectory $\wt\rho_{k+1}^{+,n}(t)$ at $(s,v+2^{-n+1})$ and follow the backward Brownian web path started from there
until it reaches $\wt\rho_k^+(t)+2^{-n}$ when we reset it to $\wt\rho_k^+(t)+2^{-n+1}$.
That is,  we set $\kappa_0^n=s$ and we let
\begin{equation}\label{defkappaj}
\kappa_{j+1}^n=\sup\{t\in[\tau_k,\kappa_j^n):\wh B_{(\kappa_j^n,\wt\rho_k^+(\kappa_j^n)+2^{-n+1})}(t)\le\wt\rho_k^+(t)+2^{-n}\}
\end{equation}
where the supremum is meant to be $-\infty$ if there is no such $t$.
Then we let
\begin{equation}\label{defrhokn}
\wt\rho_{k+1}^{+,n}(t)=\wh B_{(\kappa_j^n,\wt\rho_k^+(\kappa_j^n)+2^{-n+1})}(t)
\end{equation}
on the interval $t\in(\kappa_{j+1}^n,\kappa_j^n]$ for $j=0,1,2,\dots$ until $\kappa_{j+1}^n=-\infty$.

\begin{lemma}\label{lemma:rhodecreasing}
We have that
\begin{equation}\label{rhodecreasing}
\wt\rho_{k+1}^{+,n+1}(t)\le\wt\rho_{k+1}^{+,n}(t)
\end{equation}
for all $t\le s$.
\end{lemma}

\begin{proof}[Proof of Lemma~\ref{lemma:rhodecreasing}]
First observe that
\begin{equation}\label{rholowerbound}
\wt\rho_k^+(t)+2^{-n}\le\wt\rho_{k+1}^{+,n}(t)
\end{equation}
holds for any $t\le s$ by the definition of the stopping times $\kappa_j^n$ and that of $\wt\rho_{k+1}^{+,n}$ in \eqref{defkappaj}--\eqref{defrhokn}.
Then \eqref{rhodecreasing} can be shown as follows.
It clearly holds for $t=s$, then the inequality remains valid until $\kappa_1^{n+1}$ by the fact that both sides follow the evolution of the backward Brownian web.
At time $t=\kappa_j^{n+1}$ by definition $\wt\rho_{k+1}^{+,n+1}(\kappa_j^{n+1})=\wt\rho_k^+(\kappa_j^{n+1})+2^{-n}$
which is still a lower bound for $\wt\rho_{k+1}^{+,n}(\kappa_j^{n+1})$ by \eqref{rholowerbound} for any $j=1,2,\dots$.
This completes the proof because at $\kappa_j^n$ the left-hand side of \eqref{rhodecreasing} changes continuously
while the right-hand side has a jump of size $2^{-n}$
and at any other time the evolution of the backward Brownian web does not allow the two paths to cross each other.
\end{proof}

Hence the sequence of paths $\wt\rho_{k+1}^{+,n}(t)$ for $t\le s$ is non-increasing in $n$.
This allows us to define
\begin{equation}\label{defrhok+1}
\wt\rho_{k+1}^+(t)=\lim_{n\to\infty}\wt\rho_{k+1}^{+,n}(t).
\end{equation}
Very similarly given $\wt\rho_k^-(t)$ for $t\le s$ one defines $\wt\rho_{k+1}^-(t)$ as the limit of a non-decreasing sequence of paths.
By the obvious inequality $\wt\rho_{k+1}^-(t)\le\wt\rho_{k+1}^+(t)$ both curves are well-defined
and the limits in \eqref{defrhok+1} and in its analogue for $\wt\rho_{k+1}^-$ are finite.
Then we let
\begin{equation}
\wt\tau_{k+1}=\sup\{t\le s:\wt\rho_{k+1}^+(t)=\wt\rho_{k+1}^-(t)\}.
\end{equation}

\begin{proposition}\label{prop:rhochar}
Almost surely $\rho_k^\pm(t)=\wt\rho_k^\pm(t)$ holds for all $k=0,1,2,\dots$ and for all $t\le s$.
\end{proposition}

\begin{proof}[Proof of Proposition~\ref{prop:rhochar}]
All points of $\wt\rho_k^+$ for all $k=0,1,2,\dots$ have an incoming backward Brownian web path.
The statement is true for $k=0$ by definition, it holds by construction for $\wt\rho_{k+1}^+$ if it is away from $\wt\rho_k^+$
and at the points where they meet the statement follows by induction.
Hence all point of $\wt\rho_k^+$ are of type $(1,1)$, $(2,1)$ or $(1,2)$ in the backward Brownian web.
Note that all $(1,1)$ and $(2,1)$ points on $\wt\rho_k^+$ with a single outgoing backward Brownian web path $b$,
we have $b\le\wt\rho_k^+$ at all times that they are defined.
These statements imply that the supremum in \eqref{defrho} over all backward Brownian web paths starting
along the trajectory of $\wt\rho_k^+(r)$ for $r\in[t,s]$ equals to the supremum over the $(1,2)$ points along $\wt\rho_k^+$ only.

By definition, we have $\rho_0^+(t)=\wt\rho_0^+(t)$ for all $t\le s$.
Assume now that $\rho_k^+(t)=\wt\rho_k^+(t)$ for all $t\le s$ holds for $k$ and we prove it for $k+1$.
We first show that
\begin{equation}\label{rhotilderho}
\rho_{k+1}^+(t)\le\wt\rho_{k+1}^+(t)
\end{equation}
holds for all $t\in[\tau_k,s]$.
This follows from the inequalities $\rho_{k+1}^+(t)\le\wt\rho_{k+1}^{+,n}(t)$ which hold for all $n$ and for all $t\in[\tau_k,s]$.
They are indeed true because $\rho_{k+1}^+(t)>\wt\rho_{k+1}^{+,n}(t)$ for some $t\le s$ would mean the existence of an $r\in[t,s]$ such that $\wh B_{(r,\rho_k^+(r))+}(t)>\wt\rho_{k+1}^{+,n}(t)$,
that is, the backward Brownian web path $\wh B_{(r,\rho_k^+(r))+}$ would cross $\wt\rho_{k+1}^{+,n}$ where the latter consists of backward Brownian web path parts with possible upward jumps between the parts.
This is impossible, hence \eqref{rhotilderho}.

To prove the reverse inequality we assume that there is a $t\in[\tau_k,s]$ such that $\wt\rho_{k+1}^+(t)>\wt\rho_k^+(t)$
because otherwise the equality in \eqref{rhotilderho} holds automatically.
In case of $\wt\rho_{k+1}^+(t)>\wt\rho_k^+(t)$, the path $\wt\rho_{k+1}^+$ coincides with a backward Brownian web trajectory around $t$,
hence there are at least two different outgoing paths from the point $(t,\wt\rho_{k+1}^+(t))$ in the forward Brownian web.
We consider the lowest forward path $B_{(t,\wt\rho_{k+1}^+(t))-}$ starting at $(t,\wt\rho_{k+1}^+(t))$ which is below $\wt\rho_{k+1}^+$ locally.
By definition, the trajectory $\wt\rho_{k+1}^+$ must always follow a backward Brownian web path if it is away from $\wt\rho_k^+$.
For this reason $B_{(t,\wt\rho_{k+1}^+(t))-}$ cannot cross $\wt\rho_{k+1}^+$ as long as $\wt\rho_{k+1}^+$ is away from $\wt\rho_k^+$.
Take $\tau=\inf\{r\in[t,s]:\wt\rho_{k+1}^+(r)=\wt\rho_k^+(r)\}$ which has to exist because $\wt\rho_{k+1}^+(s)=\wt\rho_k^+(s)$.

Since $\wt\rho_{k+1}^+(r)>\wt\rho_k^+(r)$ for $r\in[t,\tau)$ by the definition of $\tau$,
the path $\wt\rho_{k+1}^+$ is equal to a backward Brownian web path on $[t,\tau)$.
Hence $\wt\rho_{k+1}^+$ itself on $[t,\tau)$ is one of the outgoing backward paths starting from $(\tau,\wt\rho_k^+(\tau))$
which implies
\begin{equation}
\wt\rho_{k+1}^+(r)\le\wh B_{(\tau,\rho_k^+(\tau))+}(r)\le\rho_{k+1}^+(r)
\end{equation}
for all $r\in[t,\tau)$ and it completes the proof.
\end{proof}

The proof of Proposition~\ref{prop:Skorokhod} is based on Proposition~\ref{prop:WfstarSkorokhod} below.
To state it, we introduce the following notation.
Let $f,g:\R_+\to\R$ be continuous functions with $f(0)=0$. The {\bf reflection of $g$ off $f$ in the Skorokhod sense}  is \begin{equation}\label{defSkorokhod}
g_{f\uparrow}(t)=g(t)-\inf_{s\in[0,t]}\left(g(s)-f(s)\right).
\end{equation}
We will apply this concept with $g=W$, a standard Brownian motion.

Next, we keep $f$ general, and we present a sequence of approximations to  $W_{f\uparrow}$ that is compatible with the Brownian web.
We use these approximations to understand the law of the paths $\rho_k^\pm=\wt\rho_k^\pm$.
In these approximations, we mimic the definition of $\wt\rho_{k+1}^+$ given $\wt\rho_k^+$.
That is, in the $n$th step we run $W(t)$ until it approaches $f(t)$ by $2^{-n}$ and then we reset it to $f(t)+2^{-n+1}$.

First, we fix $f$ and a positive integer $n$, and describe the marginal law of the approximation $W_{f,n}$.
We define $W_{f,n}(t)$ together with the sequence of stopping times $\iota_j^n$ for $j=0,1,2,\dots$ as follows.
We let $\iota_0^n=0$ and define $W_{f,n}$ to evolve as a Brownian motion started at $W_{f,n}(0)=2^{-n+1}$ until time $\iota_1^n$ defined by
\begin{equation}\label{defiota}
\iota_{j+1}^n=\inf\left\{t>\iota_j^n:W_{f,n}(t)\le f(t)+2^{-n}\right\}.
\end{equation}
For any $j=1,2,\dots$, we set $W_{f,n}(\iota_j^n)=f(\iota_j^n)+2^{-n+1}$ and we let $W_{f,n}$ evolve further
as a Brownian motion independent of the past of $W_{f,n}$.
In this way $W_{f,n}(t)$ is defined on the intervals $t\in[\iota_j^n,\iota_{j+1}^n)$ to run as Brownian motion
and with jumps at $\iota_j^n$ for $j=1,2,\dots$ from $f(\iota_j^n)+2^{-n}$ to $f(\iota_j^n)+2^{-n+1}$.

Next, we describe a specific coupling of $W$ and the $W_{f,n}$, $n=1,2,\ldots $ which is compatible with the Brownian web.
We refer to this as the {\bf Brownian web coupling}.
Under this coupling, Brownian motion paths follow a coalescing rule.
Let us sample $W$ first, which gives us $W_{f\uparrow}$.
Then we add the paths $W_{f,n}$ one by one as follows.
We start with the path $W_{f,1}$ which is distributed as a Brownian motion started at $1$ and if it reaches $f+1/2$ then it jumps to $f+1$.
We set $W_{f,1}$ to be a Brownian motion started at $1$ which is independent of $W_{f\uparrow}$ (with $0=W_{f\uparrow}(0)\le W_{f,1}(0)=1$)
until $W_{f,1}$ hits $W_{f\uparrow}$.
Then $W_{f,1}$ follows $W_{f\uparrow}$ until the time when $W_{f,1}$ jumps up from $f+1/2$ to $f+1$.
After its jump, $W_{f,1}$ evolves independently of $W_{f\uparrow}$ again until they collide.
In this way, $W_{f\uparrow}\le W_{f,1}$ holds for all times.

Suppose that $W_{f\uparrow}$ and $W_{f,1},\dots,W_{f,n}$ are already sampled and that the inequality
\begin{equation}
W_{f\uparrow}(t)\le W_{f,n}(t)\le\dots\le W_{f,1}(t)
\end{equation}
holds for all $t\ge0$.
Then we sample the trajectory $W_{f,n+1}$ which starts at $W_{f,n+1}(0)=2^{-n+2}$ above $W_{f\uparrow}(0)=0$ and below
$W_{f,1}(0)=1,\dots,W_{f,n}(0)=2^{-n+1}$.
It will be clear from the construction that this relation is kept for all time under the Brownian web coupling.
The path $W_{f,n+1}$ runs independently of those sampled so far ($W_{f\uparrow}$ and $W_{f,1},\dots,W_{f,n}$) as long as it is away from them.
Note that there are two possible collisions of $W_{f,n+1}$ with previously sampled paths.
First, if $W_{f,n+1}$ hits $W_{f,n}$ from below, then $W_{f,n+1}$ coalesces with $W_{f,n}$ and they run together
until the next jump of $W_{f,n}$ at $\iota_j^n$ for some $j$ when $W_{f,n}$ jumps up from $f(\iota_j^n)+2^{-n}$ to $f(\iota_j^n)+2^{-n+1}$
and $W_{f,n+1}$ runs further as an independent Brownian motion.
Second, if $W_{f,n+1}$ hits $W_{f\uparrow}$ from above, then $W_{f,n+1}$ coalesces with $W_{f\uparrow}$ until the next jump of $W_{f,n+1}$
at $\iota_j^{n+1}$ for some $j$ when $W_{f,n+1}$ jumps up from $f(\iota_j^{n+1})+2^{-n+1}$ to $f(\iota_j^{n+1})+2^{-n+2}$.
After $\iota_j^{n+1}$, $W_{f,n+1}$ evolves further as an independent Brownian motion again until the next collision.
In this way, $W_{f,n+1}$ can get neither below the path $W_{f\uparrow}$ nor above $W_{f,n}$.

By construction, we have
\begin{equation}
W_{f\uparrow}(t)\le W_{f,n+1}(t)\le W_{f,n}(t)
\end{equation}
for all $t\ge0$.
By monotonicity, the  limit of $W_{f,n}$ exists and it satisfies
\begin{equation}\label{Wfstarlower}
W_{f*}(t)=\lim_{n\to\infty}W_{f,n}(t)\ge W_{f\uparrow}(t)
\end{equation}
for all $t\ge0$.
The inequality in \eqref{Wfstarlower} is an equality according to the next result.

\begin{proposition}\label{prop:WfstarSkorokhod}
Let $f:\R_+\to\R$ be a continuous function with $f(0)=0$.
Then
\begin{equation}
W_{f*}(t)=W_{f\uparrow}(t)
\end{equation}
holds almost surely for all $t$.
\end{proposition}

\begin{proof}[Proof of Proposition~\ref{prop:WfstarSkorokhod}]
We define another coupling of the sequence of processes $W_{f,n}(t)$ where each of them are driven by the same Brownian motion $W(t)$.
We denote these processes by $\wt W_{f,n}(t)$ under this coupling which we define by
\begin{equation}\label{defWfn}
\wt W_{f,n}(t)=f(\iota_j^n)+2^{-n+1}+W(t)-W(\iota_j^n)
\end{equation}
for $t\in[\iota_j^n,\iota_{j+1}^n)$ where $\iota_j^n$ is given by \eqref{defiota}.
Note that for every $n$, $\wt W_{f,n}$ defined in \eqref{defWfn} has the same law as  $W_{f,n}$.

For a given integer $n$ and given $t\ge0$,
let $j$ be such that $\iota_j^n\le t<\iota_{j+1}^n$, that is $\iota_j^n$ is the last time in $[0,t]$ where $W_{f,n}$ jumps or time $0$
if there is no jump in $[0,t]$.
Next we use the property of Skorokhod reflection that for any continuous $f,g:\R_+\to\R$ and $0\le s\le r$ the inequality
$g(r)-g(s)\le g_{f\uparrow}(r)-g_{f\uparrow}(s)$ holds.
For this reason, the increment of $W$ on the right-hand side of \eqref{defWfn} can be upper bounded almost surely as
\begin{equation}\label{Wfnbound}
\wt W_{f,n}(t)\le f(\iota_j^n)+2^{-n+1}+W_{f\uparrow}(t)-W_{f\uparrow}(\iota_j^n).
\end{equation}
Another property of Skorokhod reflection is that $W_{f\uparrow}(\iota_j^n)\ge f(\iota_j^n)$,
hence \eqref{Wfnbound} can be further upper bounded as
\begin{equation}
\wt W_{f,n}(t)\le2^{-n+1}+W_{f\uparrow}(t)
\end{equation}
which holds almost surely for all $t$ uniformly.
Also, by construction $W_{f\uparrow}\le \wt W_{f,n}$. This implies that $\sup_t|\wt W_{f,n}(t)-W_{f\uparrow}(t)|\to 0$,
and hence the sequence $\wt W_{f,n}$ is tight in the sup-norm topology.
So $W_{f,n}$ is also tight since it has the same distribution.
In particular, its pointwise limit $W_{f*}$ is continuous.
So for every $t\ge 0$, almost surely
\begin{equation}
W_{f\uparrow}(t) \stackrel{\d}{=} W_{f*}(t)\ge W_{f\uparrow}(t),
\end{equation}
and so $W_{f\uparrow}(t)=W_{f*}(t)$.
The continuity of $W_{f*}$ and $W_{f\uparrow}$ now implies the claim.
\end{proof}

\begin{proof}[Proof of Proposition~\ref{prop:Skorokhod}]
By definition \eqref{defrho} the curve $\rho_{k+1}^+$ is independent of the randomness above $\rho_{k+1}^+$
which implies that the sequence $\rho_k^\pm(t)$ on $[\tau_k,s]$ is Markovian in $k$.
Hence the distribution of $\rho_{k+1}^+$ only depends on $\rho_k^+$.
It remains to show that $\rho_{k+1}^+$ is a Skorokhod reflected Brownian motion off $\rho_k^+$ on $[\tau_k,s]$, the evolution on $[\tau_{k+1},\tau_k]$ is clear.

By Proposition~\ref{prop:rhochar} it holds that $\rho_{k+1}^+=\wt\rho_{k+1}^+$.
We show next that $\wt\rho_{k+1}^+$ is a Brownian motion reflected off $\wt\rho_k^+$ in the Skorokhod sense.
Instead of specifying the backward Brownian motion which is reflected off $\wt\rho_k^+(t)$ to construct $\wt\rho_{k+1}^+(t)$
we use Proposition~\ref{prop:WfstarSkorokhod} backward in time with the continuous curve $f(t)=\wt\rho_k^+(s-t)$ for $t\ge0$.
Then $W_{f,n}(t)=\wt\rho_{k+1}^{+,n}(s-t)$ holds for all integer $n$ and $t\ge0$ under the Brownian web coupling of $W_{f,n}$.
Proposition~\ref{prop:WfstarSkorokhod} yields that the almost sure decreasing limit $\wt\rho_{k+1}^+$
is a Brownian path reflected off $\wt\rho_k^+$ in the Skorokhod sense which completes the proof.
\end{proof}

\begin{proof}[Proof of Proposition~\ref{prop:regions}]
By Proposition~\ref{prop:rhochar}, we have $\rho_k^\pm=\wt\rho_k^\pm$ for all $k$, hence it is enough to show that the curves $\wt\rho_k^\pm$ serve as the boundary of the region $Q_k$.
We proceed by induction on $k$.
The statement clearly holds for $k=0$ because the points $\{(t,x):D^{\Br,\LSC}(t,x;s,I)=0\}$ are exactly those for which $B_{(t,x)}(s)\in I$
and these points are the ones between $\wt\rho_0^-(t)=\wh B_{(s,u)}(t)$ and $\wt\rho_0^+(t)=\wh B_{(s,v)}(t)$ for $t\in[\tau_0,s]$.

Next we assume that
\begin{equation}\label{Qkwithrho}
Q_k=\left\{(t,x):t\in[\tau_k,s],x\in[\wt\rho_k^-(t),\wt\rho_k^+(t)]\right\}
\end{equation}
holds for some $k$ and we prove it for $k+1$ below.
We choose an arbitrary $(t,x)\not\in Q_k$ and we prove that if the distance $D^{\Br,\LSC}(t,x;s,I)$ is $k+1$,
then the point $(t,x)$ belongs to the right-hand side of \eqref{Qkwithrho}
and if the distance is larger than $k+1$, then $(t,x)$ is not in the right-hand side of \eqref{Qkwithrho}.

If the forward Brownian web path $B_{(t,x)}$ hits the boundary of the set $Q_k$, then by definition $D^{\Br,\LSC}(t,x;s,I)=k+1$ and $(t,x)\in Q_{k+1}$.
Without loss of generality we can assume that $B_{(t,x)}$ hits the curve $\wt\rho_k^+$ at some $(t^*,\wt\rho_k^+(t^*))$.
We show that the point $(t,x)$ is below the curve $\wt\rho_{k+1}^+$, that is, $x\le\wt\rho_{k+1}^+(t)$
which by definition \eqref{defrhok+1} means that $x\le\wt\rho_{k+1}^{+,n}(t)$ for all $n$.
This is however true because $B_{(t,x)}(r)$ is a continuous forward trajectory for $r\in[t,t^*]$
starting at $B_{(t,x)}(t)=x$ and ending at $B_{(t,x)}(t^*)=\wt\rho_k^+(t^*)$.
The backward trajectory $\wt\rho_{k+1}^{+,n}(r)$ for $r\in[t,t^*]$ starts at $\wt\rho_{k+1}^{+,n}(t^*)\ge\wt\rho_k^+(t^*)+2^{-n}$, that is, above the endpoint of $B_{(t,x)}$.
Furthermore $\wt\rho_{k+1}^{+,n}(r)$ is defined to follow a backward Brownian web trajectory on the intervals $(\kappa_{j+1}^n,\kappa_j^n]$
and to jump up by $2^{-n}$ at every $\kappa_j^n$.
For this reason $\wt\rho_{k+1}^{+,n}(r)$ can never cross $B_{(t,x)}(r)$ and it remains above $B_{(t,x)}(r)$ for $r\in[t,t^*]$.
In particular $x=B_{(t,x)}(t)\le\wt\rho_{k+1}^{+,n}(t)$ for all $n$ which proves that $x\le\wt\rho_{k+1}^+(t)$.

If the path $B_{(t,x)}$ avoids the set $Q_k$, then $D^{\Br,\LSC}(t,x;s,I)>k+1$ and $(t,x)\not\in Q_{k+1}$.
In this case we may assume that $B_{(t,x)}$ passes above the continuous upper boundary $\wt\rho_k^+$ (continuity follows from Proposition~\ref{prop:Skorokhod}).
Then $B_{(t,x)}$ has a positive distance from $\wt\rho_k^+$, that is, there is an $\varepsilon>0$ such that $B_{(t,x)}(r)\ge\wt\rho_k^+(r)+\varepsilon$ for all $r\in[t,s]$.
Let $n\ge1-\log_2\varepsilon$ which means that $2^{-n+1}\le\varepsilon$.
We show that the point $(t,x)$ is above $\wt\rho_{k+1}^{+,n}$, that is, $x\ge\wt\rho_{k+1}^{+,n}(t)\ge\wt\rho_{k+1}^+(t)$ where the second inequality holds due to the fact that $\wt\rho_{k+1}^+$ is the decreasing limit of $\wt\rho_{k+1}^{+,n}$.
The inequality $x\ge\wt\rho_{k+1}^{+,n}(t)$ is true because $B_{(t,x)}(r)$ is a continuous forward trajectory for $r\in[t,s]$ starting at $B_{(t,x)}(t)=x$ and ending at $B_{(t,x)}(s)\ge v+\varepsilon$.
The backward trajectory $\wt\rho_{k+1}^{+,n}(r)$ for $r\in[t,s]$ starts at $v+2^{-n+1}$ which is below the endpoint of $B_{(t,x)}$.
As long as $\wt\rho_{k+1}^{+,n}$ follows a backward Brownian web path it cannot cross $B_{(t,x)}$.
In the points $\kappa_j^n$ the path $\wt\rho_{k+1}^{+,n}$ jumps from $\wt\rho_k^++2^{-n}$ to $\wt\rho_k^++2^{-n+1}$ which is still below $\wt\rho_k^++\varepsilon\le B_{(t,x)}$.
Therefore $\wt\rho_{k+1}^{+,n}$ remains below $B_{(t,x)}$, hence $x\ge\wt\rho_{k+1}^{+,n}(t)\ge\wt\rho_{k+1}^+(t)$.
\end{proof}

\subsection{Properties of discrete region boundaries}
\label{ss:discreteproperties}

This subsection contains the proofs of Propositions~\ref{prop:discreteSkorokhod} and \ref{prop:discreteregions}.
about the boundaries of regions with different discrete web distances.
First we prove the following discrete analogue of Proposition~\ref{prop:rhochar} which describes the evolution of the discrete region boundaries.
Heuristically $r_{k+1}^+$ follows the evolution of a backward discrete Brownian web trajectory of $\wh Y$ as long as it is away from $r_k^+$, cf.~\eqref{defYhat},
and $r_{k+1}^+$ is forced to jump up if it is equal to $r_k^+$ until time $T_k$.
Beyond $T_k$, $r_{k+1}^+$ follows a backward discrete Brownian web trajectory.
The evolution of $r_{k+1}^-$ given $r_k^-$ can be given similarly.

\begin{proposition}\label{prop:rchar}
For any $k=0,1,2,\dots$, conditionally given the trajectory of $r_k^+$ and the time $T_k$ when $r_k^\pm$ meet,
the evolution of $r_{k+1}^+$ is given by
\begin{equation}\label{rk+1evolution}
r_{k+1}^+(i-1)=\left\{\begin{array}{ll}r_{k+1}^+(i)-\xi_{(i-1,r_{k+1}^+(i))} & \mbox{if}\quad r_{k+1}^+(i)>r_k^+(i)\\
r_{k+1}^+(i)+1 & \mbox{if}\quad r_{k+1}^+(i)=r_k^+(i)\end{array}\right.
\end{equation}
for all $i=j,j-1,\dots,T_k+1$.
For times $i=T_k,T_k-1,\dots$ we have that
\begin{equation}\label{rk+1evolution2}
r_{k+1}^+(i-1)=r_{k+1}^+(i)-\xi_{(i-1,r_{k+1}^+(i))}.
\end{equation}
\end{proposition}

\begin{proof}[Proof of Proposition~\ref{prop:rchar}]
Given the path $r_k^+$ up to time $T_k$ let $\wt r_{k+1}^+$ denote the trajectory given by the evolution rules \eqref{rk+1evolution}--\eqref{rk+1evolution2}.
We prove that for any $k=0,1,2,\dots$ given $r_k^+$ it holds that $r_{k+1}^+(i)=\wt r_{k+1}^+(i)$ for all $i\le j$ by induction backwards on $i$.
The case $i=j$ is true by definition.
We assume that $r_{k+1}^+$ and $\wt r_{k+1}^+$ agree for $i+1,i+2,\dots,j$.
To see the equality $r_{k+1}^+(i)=\wt r_{k+1}^+(i)$ we rewrite \eqref{defr} by separating the $l=i$ term as
\begin{equation}\label{rk+1max}\begin{aligned}
r_{k+1}^+(i)&=\max\left(\wh Y_{(i,r_k^+(i+1)+1)}(i),\max_{l\in\{i+1,\dots,j\}}\wh Y_{(l,r_k^+(l+1)+1)}(i)\right)\\
&=\max\left(r_k^+(i+1)+1,r_{k+1}^+(i+1)-\xi_{(i,r_{k+1}^+(i+1))}\right).
\end{aligned}\end{equation}

Since $r_{k+1}^+\ge r_k^+$ always holds there are two possibilities: either $r_{k+1}^+(i+1)=r_k^+(i+1)$ or $r_{k+1}^+(i+1)>r_k^+(i+1)$.
If $r_{k+1}^+(i+1)=r_k^+(i+1)$, then $r_{k+1}^+(i)=r_k^+(i+1)+1$ since the first term in the maximum on the right-hand side of \eqref{rk+1max} cannot be smaller than the other one.
Also $\wt r_{k+1}^+(i)=r_{k+1}^+(i+1)+1=r_k^+(i+1)+1$ by \eqref{rk+1evolution} hence $r_{k+1}^+(i)=\wt r_{k+1}^+(i)$.
If $r_{k+1}^+(i+1)>r_k^+(i+1)$, then $r_{k+1}^+(i+1)\ge r_k^+(i+1)+2$, hence $r_{k+1}^+(i)=r_{k+1}^+(i+1)-\xi_{(i,r_{k+1}^+(i+1))}$ because the second term in the maximum on the right-hand side of \eqref{rk+1max} cannot be smaller than the first one.
On the other hand by \eqref{rk+1evolution} we have $\wt r_{k+1}^+(i)=r_{k+1}^+(i+1)-\xi_{(i,r_{k+1}^+(i+1))}$
which means that $r_{k+1}^+(i)=\wt r_{k+1}^+(i)$ as required for $i\ge T_k$.
The evolution beyond $T_k$ clearly follows that of the backward Brownian web given in \eqref{rk+1evolution2}.
\end{proof}

\begin{proof}[Proof of Proposition~\ref{prop:discreteSkorokhod}]
We prove the statement for the upper boundary curves $r_k^+$, the proof for $r_k^-$ is identical.
To see \eqref{rkSkorokhod} we proceed by induction on $k$.
Given $r_k^+$, the path $r_{k+1}^+$ clearly starts at $(j,v+1)$ as well as the right-hand side of \eqref{rkSkorokhod}.
Next we check that the right-hand side of \eqref{rkSkorokhod} satisfies the evolution rule \eqref{rk+1evolution} by induction on $i$ backwards.
For this we rewrite the right-hand side of \eqref{rkSkorokhod} by separating the $l=i$ case in the minimum to get that
\begin{equation}\label{sk+1computation}\begin{aligned}
&s_{k+1}(i)-\min_{l\in\{i,\dots,j\}}\left(s_{k+1}(l)-r_k^+(l+1)-1\right)\\
&\quad=s_{k+1}(i)-\min\left(s_{k+1}(i)-r_k^+(i+1)-1,\min_{l\in\{i+1,\dots,j\}}\left(s_{k+1}(l)-r_k^+(l+1)-1\right)\right)\\
&\quad=s_{k+1}(i)-\min\left(s_{k+1}(i)-r_k^+(i+1)-1,s_{k+1}(i+1)-r_{k+1}^+(i+1)\right)\\
&\quad=\max\left(r_k^+(i+1)+1,s_{k+1}(i)-s_{k+1}(i+1)+r_{k+1}^+(i+1)\right)
\end{aligned}\end{equation}
where the second equality above follows by the induction hypothesis \eqref{rkSkorokhod} with $i$ replaced by $i+1$.
The right-hand side of \eqref{sk+1computation} means an independent random walk step compared to $r_{k+1}^+(i+1)$ except for the case when $r_{k+1}^+(i+1)=r_k^+(i+1)$ in which case it is forced to jump up by one.
This corresponds to the evolution of $r_{k+1}^+(i)$ described in \eqref{rk+1evolution} hence it proves \eqref{rkSkorokhod}.

The proof of \eqref{rkwithLPP} follows an induction on $k$.
It clearly holds for $k=0$ and the induction step is based on \eqref{rkSkorokhod} which one can write as
\begin{equation}\begin{aligned}
r_k^+(i)&=\max_{l\in\{i,\dots,j\}}\left(s_k(i)-s_k(l)+r_{k-1}^+(l+1)+1\right)\\
&=\max_{l\in\{i,\dots,j\}}\Big(s_k(i)-s_k(l)\\
&\qquad+\max_{l+1\le l_{k-1}\le\dots\le l_{-1}=j-k+1}\sum_{m=0}^{k-1}\left(s_m(l_m)-s_m(l_{m-1})\right)+k+v+1\Big)
\end{aligned}\end{equation}
where we used the induction hypothesis \eqref{rkwithLPP} for $r_{k-1}^+(l+1)$ in the last equality above.
Then the right-hand side above is equal in distribution to that of \eqref{rkwithLPP} which can be seen by shifting all indices $l_{k-1},\dots,l_{-1}$ by $1$.
\end{proof}

\begin{proof}[Proof of Proposition~\ref{prop:discreteregions}]
We proceed by induction on $k$.
First we prove the statement for $k=0$.
Take any $i\in[T_0,j]$ and $m\in\Z$ so that $(i,m)\in\Z^2_\e$.
The forward random walk web path $Y_{(i,m)}$ cannot cross either of the backward paths
$r_0^+=\wh Y_{(j,v+1)}$ and $r_0^-=\wh Y_{(j,u-1)}$ on the interval $[i,j]$.
Hence if $r_0^-(i)<m<r_0^+(i)$, then $Y_{(i,m)}(j)\in I$ and $D^\Be(i,m;j,I)=0$.
If $m<r_0^-(i)$ or $m>r_0^+(i)$, then $Y_{(i,m)}(j)\not\in I$ and $D^\Be(i,m;j,I)>0$.
Also if $i<T_0$, then the forward random walk web path $Y_{(i,m)}$ avoids $j\times I$ and $D^\Be(i,m;j,I)>0$.

Assume that the statement is true for $k$, that is,
$D^\Be(i,m;j,I)\le k$ for points $(i,m)\in\Z^2_\e$ between $r_k^+$ and $r_k^-$ on $(T_k,j]$ and $D^\Be(i,m;j,I)>k$ for all other points $(i,m)\in\Z^2_\e$.
Now we prove it for $k+1$ by another induction on $i$ backwards in time as follows.
The statement of the proposition clearly holds for $k+1$ and $i=j$.
Indeed we have that $D^\Be(j,m;j,I)=0\le k+1$ exactly if $r_{k+1}^-(j)=u-1<m<r_{k+1}^+(j)=v+1$, otherwise $D^\Be(j,m;j,I)=\infty>k+1$.

Next we assume that the statement holds for $k+1$ and $i\in[T_k,j]$ and we prove it for $k+1$ and $i-1$ in what follows.
By assumption around the upper boundary the points $(i,m)\in\Z^2_\e$ with $r_k^+(i)<m<r_{k+1}^+(i)$ have $D^\Be(i,m;j,I)=k+1$.
We have two cases depending on weather there are such $m$s or not, that is $r_k^+(i)<r_{k+1}^+(i)$ or $r_k^+(i)=r_{k+1}^+(i)$.

We deal with the case $r_k^+(i)<r_{k+1}^+(i)$ first.
Then $D^\Be(i-1,m;j,I)\le k+1$ certainly holds for all points $(i-1,m)\in\Z^2_\e$ with $r_k^+(i-1)<m\le r_{k+1}^+(i)-2$
because these points $(i-1,m)$ are connected to either of $(i,m\pm1)$ by an edge in the random walk web $Y$ and $D^\Be(i,m\pm1;j,I)\le k+1$ by assumption.
If $\xi_{(i-1,r_{k+1}^+(i))}=-1$, then there is an edge in $Y$ from $(i-1,r_{k+1}^+(i))$ to $(i,r_{k+1}^+(i)-1)$,
hence $D^\Be(i-1,r_{k+1}^+(i);j,I)=D^\Be(i,r_{k+1}^+(i)-1)=k+1$ and $(i-1,r_{k+1}^+(i))\in R_{k+1}$.
If $\xi_{(i-1,r_{k+1}^+(i))}=1$, then there is an edge in $Y$ from $(i-1,r_{k+1}^+(i))$ to $(i,r_{k+1}^+(i)+1)$,
hence $D^\Be(i-1,r_{k+1}^+(i);j,I)=D^\Be(i,r_{k+1}^+(i)-1)+1=k+2$ and $(i-1,r_{k+1}^+(i))\not\in R_{k+1}$.
This means that the boundary of $R_{k+1}$ changes by $-\xi_{(i-1,r_{k+1}^+(i))}$ in the step $i\to i-1$
which is the same as the evolution of $r_{k+1}^+$ in the first case in \eqref{rk+1evolution}.

In the case $r_k^+(i)=r_{k+1}^+(i)$ regardless of the value of $\xi_{(i-1,r_k^+(i))}$ and whether the edge from $(i-1,r_k^+(i))$ to $(i,r_k^+(i)-1)$ is in $Y$ or not,
we can bound $D^\Be(i-1,r_k^+(i);j,I)\le D^\Be(i,r_k^+(i)-1;j,I)+1\le k+1$, hence $(i-1,r_k^+(i))\not\in R_{k+1}$.
This means that the boundary of $R_{k+1}$ increases by $1$ in this step
and it corresponds to the second case in \eqref{rk+1evolution} where $r_{k+1}^+(i-1)=r_{k+1}^+(i)+1$.

For $i\le T_k$, we know by assumption that $D^\Be(i,m;j,I)=k+1$ if $(i,m)\in\Z^2_\e$ with $r_{k+1}^+(i)<m<r_{k+1}^-(i)$
and $D^\Be(i,m;j,I)>k+1$ for all other values of $m$.
Hence the boundary of $R_{k+1}$ changes by $-\xi_{(i-1,r_{k+1}^\pm(i))}$
which means that the boundary follows the same backward random walk web trajectory as $r_{k+1}^\pm$, cf.~\eqref{rk+1evolution2}.
\end{proof}

\subsection{Further properties of the Brownian web distance}
\label{ss:furtherprop}

We prove Propositions~\ref{prop:DBraltern}, \ref{prop:DBrLSC}, \ref{prop:DBrfiniteness} and \ref{prop:DBrrational}
about further properties of the Brownian web distance in this subsection.

\begin{proof}[Proof of Proposition~\ref{prop:DBraltern}]
The function $D^\Br$ satisfies the triangle inequality and the distance of any point from itself is $0$, hence $D^\Br$ is a directed metric.
Since $D^\Br$ agrees with $d$ on the the set where $d$ is defined, the induced directed metric $\wt D^\Br$ exists and the inequality
\begin{equation}\label{DDtildeineq}
D^\Br(t,x;s,y)\le\wt D^\Br(t,x;s,y)
\end{equation}
holds for any $(t,x,s,y)\in\R^4$ by Definition~\ref{def:DBrtilde}.
Furthermore, if $(s,y)$ is on one of the outgoing paths started at $(t,x)$, then we also have by definition that
\begin{equation}\label{Ddineq}
\wt D^\Br(t,x;s,y)\le d(t,x;s,y).
\end{equation}

In order to show that \eqref{DDtildeineq} holds with an equality, we choose $(t,x,s,y)\in\R^4$ such that $D^\Br(t,x;s,y)=k$ for some finite integer $k$.
By Definition~\ref{def:DBr}, there are points $(t_1,x_1),\dots,(t_k,x_k)$
at which the optimal path from $(t,x)$ to $(s,y)$ switches between different Brownian web trajectories.
If $t_1>t_0=t$ then $(t_1,x_1)$ is on $B_{(t,x)}$ hence $d(t_0,x_0;t_1,x_1)=0$.
The distances between further points along the optimal path are
\begin{equation}\label{d=1}
d(t_1,x_1;t_2,x_2)=\dots=d(t_k,x_k;t_{k+1},x_{k+1})=1.
\end{equation}
By using the triangle inequality of $\wt D^\Br$ along the optimal path and \eqref{Ddineq}
we get that $\wt D^\Br(t,x;s,y)=\wt D^\Br(t_0,x_0;t_{k+1},x_{k+1})\le k$.
If $t_1=t_0=t$ then by the same argument we can conclude that $\wt D^\Br(t,x;s,y)=\wt D^\Br(t_1,x_1;t_{k+1},x_{k+1})\le k$
Hence in both cases $\wt D^\Br(t,x;s,y)\le k$ holds proving equality in \eqref{DDtildeineq}.
\end{proof}

\begin{proof}[Proof of Proposition~\ref{prop:DBrLSC}]
By definition of lower semicontinuity we have to show that for any convergent sequence $(t^{(n)},x^{(n)},s^{(n)},y^{(n)})\to(t,x,s,y)\in\R^4$
it holds that
\begin{equation}\label{DBrLSC}
\liminf_{n\to\infty}D^{\Br,\LSC}(t^{(n)},x^{(n)};s^{(n)},y^{(n)})\ge D^{\Br,\LSC}(t,x;s,y).
\end{equation}
We can assume that the left-hand side of \eqref{DBrLSC} is finite and let its value be equal to $k$.
By taking a subsequence we may assume that the limit on the left-hand side exists and it is equal to $k$.
This means that for all $n$ large enough there are $k+1$ paths $\pi_0^{(n)},\dots,\pi_k^{(n)}$ in the Brownian web $B$
along the geodesic between $(t^{(n)},x^{(n)})$ and $(s^{(n)},y^{(n)})$ with $k$ jumps.
Since $B$ as an element of $H$ is a compact collection of paths in $(\Pi,d)$,
the sequence $\pi_j^{(n)}$ must have a subsequential limit $\pi_j^{(\infty)}$ for all $j=0,\dots,k$.
One can find the limiting paths $\pi_j^{(\infty)}$ so that they realize a geodesic for $D^{\Br,\LSC}$ between $(t,x)$ and $(s,y)$,
that is, $\pi_j^{(\infty)}$ starts from a point on $\pi_{j-1}^{(\infty)}$ for all $j=1,\dots,k$.
This shows that $D^{\Br,\LSC}(t,x;s,y)\le k$ proving the lower semicontinuity \eqref{DBrLSC}.
\end{proof}

Next we prove Proposition~\ref{prop:DBrfiniteness} about the finiteness of the Brownian web distance.
We recall the Brownian last passage percolation with general boundary condition $f$ defined in \eqref{defBLPPwithboundary}.
By Lemma~\ref{lemma:BLPPSkorokhod}, the Brownian last passage percolation with boundary $f$ can be obtained as an iterated Skorokhod reflection of Brownian paths started from $f$.
Next we show that the Brownian last passage percolation diverges everywhere.

\begin{lemma}\label{lemma:BLPPtoinfty}
Let $f:\R_+\to\R$ be a fixed continuous function.
For the Brownian last passage percolation with boundary condition $f$ it holds that for any $t>0$ fixed
\begin{equation}\label{BLPPtoinfty}
\P\left(\lim_{n\to\infty}L^f(t,n)=+\infty\right)=1.
\end{equation}
\end{lemma}

\begin{proof}
The right-hand side in the definition \eqref{defBLPPwithboundary} is lower bounded by
\begin{equation}
\sup_{0=t_0\le t_1\le\dots\le t_n=t}\left(f(t_1)+\sum_{i=2}^n(W_i(t_i)-W_i(t_{i-1}))\right)\ge\max_{i=2,\dots,n}(W_i(t))\ge\max_{i=2,\dots,n}(W_i(t)\wedge M)
\end{equation}
for any $M\in\R$.
Let $\varepsilon>0$ be arbitrary and note that
\begin{equation}
\P(\max_{i=2,\dots,n}(W_i(t)\wedge M)<M-\varepsilon)=\P(W_2(t)<M-\varepsilon)^n\to0
\end{equation}
as $n\to\infty$, that is $\max_{i=2,\dots,n}(W_i(t)\wedge M)$ converges to $M$ in distribution and in probability because the limit is a constant.
Since the sequence $\max_{i=2,\dots,n}(W_i(t)\wedge M)$ is non-decreasing in $n$, there is an almost sure convergence.
This implies by \eqref{defBLPPwithboundary} that $\liminf_{n\to\infty}L^f(t,n)\ge M$ almost surely for any $M\in\R$ which completes the proof.
\end{proof}

\begin{proof}[Proof of Proposition~\ref{prop:DBrfiniteness}]
Assume that $(s,y)$ is an interior point of the path $B_{(t,x)}$.
Then the point $(s,y)$ can be of type $(1,1)$ or $(2,1)$ in the forward Brownian web $B$ and of type $(0,2)$ or $(0,3)$ in the dual $\wh B$,
that is, there are two or three dual paths starting at $(s,y)$.

We consider the regions $\wt Q_k=\{(u,z):D^{\Br,\LSC}(t,x;s,y)\le k\}$ for $k=0,1,2,\dots$ which are the analogues of $Q_k$ defined in \eqref{defDBrI}
but with the interval $s\times I$ replaced by the single point $(s,y)$.
Let $\wt\rho_k^\pm$ denote the upper and lower boundaries of the regions $\wt Q_k$ which can be obtained similarly to $\rho_k^\pm$ that are the boundaries for $Q_k$,
see Proposition~\ref{prop:regions}.
The boundaries $\wt\rho_0^\pm$ on the interval $[t,s]$ are $\wh B_{(s,y)\pm}$, that is, the highest and lowest dual Brownian web paths starting at $(s,y)$.
The boundaries $\wt\rho_k^\pm$ for $k=1,2,\dots$ on $[t,s]$ can be given inductively and they are Brownian motions reflected off in the Skorokhod sense
upwards and downwards from $\wt\rho_{k-1}^\pm$.

By letting $f(v)=\wt\rho_0^+(s-v)$ for $v\in[0,s-t]$, Lemma~\ref{lemma:BLPPSkorokhod} implies that the distribution of $(\wt\rho_k^+(s-v),v\in[0,s-t])$
is equal to that of Brownian last passage percolation $(L^f(v,k),v\in[0,s-t])$ with boundary condition $f$.
By \eqref{BLPPtoinfty} in Lemma~\ref{lemma:BLPPtoinfty}, for any $u\in[t,s)$, $\wt\rho_k^+(u)\to\infty$ as $k\to\infty$ almost surely
and very similarly $\wt\rho_k^-(u)\to-\infty$.
Hence almost surely for any $(u,z)$ with $u\in[t,s)$ and $z\in\R$, we have $D^{\Br,\LSC}(u,z;s,y)<\infty$.
If $u<t$, one simply follows the Brownian web path $B_{(u,z)}$ until time $t$ to reach a point with finite Brownian distance to $(s,y)$.

In the case when $(s,y)$ is not the interior point of any Brownian web path, that is, it is not on the skeleton of the Brownian web $B$,
then there is no incoming trajectory of $B$ in this point, hence by Definition~\ref{def:DBr}, $D^{\Br,\LSC}(t,x;s,y)$ cannot be finite for any $(t,x)\in\R^2$.
\end{proof}

\begin{lemma}\label{lemma:LSCinterval}
Let $f:\R\to\R$ be a lower semicontinuous function.
We extend it to intervals $I\subset\R$ as $f(I)=\inf_{y\in I}f(y)$.
Then it holds for any $y\in\R$ that
\begin{equation}\label{LSCinterval}
f(y)=\sup_{I\subset\R:y\in I}f(I).
\end{equation}
\end{lemma}

\begin{proof}[Proof of Lemma~\ref{lemma:LSCinterval}]
Fix $y\in\R$.
Then $f(y)\ge\sup_{I\subset\R:y\in I}f(I)$ clearly holds because $f(y)\ge f(I)$ if $y\in I$.
The strict inequality $f(y)>\sup_{I\subset\R:y\in I}f(I)$ would mean that there is an $\varepsilon>0$
such that $f(I)\le f(y)-\varepsilon$ for all $I$ with $y\in I$.
By letting $y_n$ be an approximate minimizer of $f$ in the interval $I_n$ so that $y\in I_n$ and $|I_n|\to0$,
one would get a sequence for which $\liminf_{y_n\to y}f(y_n)\le f(y)-\varepsilon/2<f(y)$ which contradicts the lower semicontinuity of $f$.
\end{proof}

\begin{proof}[Proof of Proposition~\ref{prop:DBrrational}]
We assume that the values of $D^\Br(t,x;s,I)$ are given for all rational $t,x,s,u,v$ where $I=[u,v]$.
We show in several steps that it uniquely extends to $\R^4$.
We proceed using the lower semicontinuous version $D^{\Br,\LSC}(t,x;s,y)$ which agrees with $D^\Br(t,x;s,y)$ for almost all starting points $(t,x)\in\R^2$.

First we prove that the values of $D^{\Br,\LSC}(t,x;s,I)$ are determined for all $t,x\in\R$ and for all rational $s,u,v$.
We fix a rational $s$ and an interval $I=[u,v]$ with rational endpoints.
By the characterization of regions with distance $k$ from a fixed $s$ and $I$ given in Proposition~\ref{prop:regions},
it holds that for any $t\le s$ the set
\begin{equation}\label{rhointerval}
\{x\in\R:D^{\Br,\LSC}(t,x;s,I)\le k\}=[\rho_k^-(t),\rho_k^+(t)]
\end{equation}
is an interval.
Since $D^\Br(t,x;s,y)=D^{\Br,\LSC}(t,x;s,y)$ for almost all $(t,x)$ we have that
\begin{equation}
\sup\{x\in\R:D^\Br(t,x;s,I)\le k\}=\sup\{x\in\R:D^{\Br,\LSC}(t,x;s,I)\le k\}=\rho_k^+(t)
\end{equation}
and a similar equation holds for $\rho_k^-(t)$.
This implies that knowing the values $D^\Br(t,x;s,I)$ for fixed $s$ and $I$ and for $(t,x)$ from a dense subset of $\R^2$
settles the continuous curves $\rho_k^\pm$ for all $k=0,1,\dots$ by \eqref{rhointerval}.
Hence the values of $D^{\Br,\LSC}(t,x;s,I)$ are determined for all $(t,x)\in\R^2$.

Second we show that the values of $D^{\Br,\LSC}(t,x;s,y)$ are determined for all $t,x,y\in\R$ and for all rational $s$.
It is a consequence of Proposition~\ref{prop:DBrLSC} that $y\mapsto D^{\Br,\LSC}(t,x;s,y)$ is lower semicontinuous for any $t,x,s$.
By applying Lemma~\ref{lemma:LSCinterval} it holds that
\begin{equation}\label{DBrassup}
D^{\Br,\LSC}(t,x;s,y)=\sup_{I\subset\R:y\in I}D^{\Br,\LSC}(t,x;s,I).
\end{equation}
The supremum on the right-hand side of \eqref{DBrassup} does not change
if we replace it with the supremum over all intervals $I$ with rational endpoints which contain $y$.
This proves that $D^{\Br,\LSC}(t,x;s,y)$ for $t,x,y\in\R$ and $s$ rational are determined.

Next we prove that $D^{\Br,\LSC}(t,x;s,y)$ are determined for all $(t,x;s,y)\in\R^4$.
We claim that for any $(t,x;s,y)\in\R^4$, $D^{\Br,\LSC}(t,x;s,y)\le k$ holds if and only if
there is an increasing sequence of rationals $s_n$ converging to $s$ and a sequence $y_n\to y$ such that $D^{\Br,\LSC}(t,x;s_n,y_n)\le k$.
The claim is seen as follows.

If $D^{\Br,\LSC}(t,x;s,y)\le k$, then there is a last jump between different Brownian web paths along the geodesic from $(t,x)$ to $(s,y)$
which we call $(t_k,x_k)$ for simplicity.
We choose any increasing sequence of rationals $s_n$ from the interval $(t_k,s)$ so that $s_n\to s$
and we let $y_n=B_{(t_k,x_k)}(s_n)$ to be the value of the Brownian web path along the geodesic at $s_n$.
Then $y_n\to y$ and $D^{\Br,\LSC}(t,x;s_n,y_n)\le k$ holds.
Conversely, if a sequence $(s_n,y_n)\to(s,y)$ exists with the desired properties,
then lower semicontinuity implies that
\begin{equation}
D^{\Br,\LSC}(t,x;s,y)\le\liminf_{n\to\infty}D^{\Br,\LSC}(t,x;s_n,y_n)\le k
\end{equation}
which proves the claim.
The claim implies that the values of $D^{\Br,\LSC}(t,x;s,y)$ for all $(t,x;s,y)\in\R^4$ can be determined.

Finally we argue that all values of $D^{\Br,\LSC}(t,x;s,y)$ determine those of $D^\Br(t,x;s,y)$.
The type of the point $(t,x)$ in the Brownian web can be obtained as follows.
The point $(t,x)$ has an incoming path if there is an $\varepsilon>0$ and $w\in\R$ such that $D^{\Br,\LSC}(t-\varepsilon,w;t,x)=0$.
The number of outgoing paths starting at $(t,x)$ is $k=1,2,3$
if there are $k$ different $z_1,\dots,z_k$ so that for any $\varepsilon>0$ small enough $D^{\Br,\LSC}(t,x;t+\varepsilon,z_i)=0$ holds for all $i=1,\dots,k$.
If $(t,x)$ is a type $(1,2)$ point then $B_{(t,x)}$ can be determined by the property that for some small $\varepsilon>0$ it holds that
$D^{\Br,\LSC}(t-\varepsilon,w;t+\varepsilon,z)=0$ for a point $(t-\varepsilon,w)$ on the incoming path to $(t,x)$ and with $(t+\varepsilon,z)$ on $B_{(t,x)}$.
This information is enough to determine $D^\Br(t,x;s,y)$, see Definitions~\ref{def:DBr} and \ref{def:DBrtilde}.
The proof is complete.
\end{proof}

\paragraph{Acknowledgments.}
We thank Dor Elboim and Ron Peled for a question leading to Theorem~\ref{thm:DBehorizontal}
and Julian Ransford for pointing out the short proof of Proposition~\ref{prop:DBrLSC}.
The work of B.\ Vet\H o was supported by the NKFI (National Research, Development and Innovation Office)
grants FK142124 and KKP144059 ``Fractal geometry and applications'',
by the Bolyai Research Scholarship of the Hungarian Academy of Sciences
and by the \'UNKP--22--5--BME--250 New National Excellence Program of the Ministry for Innovation and Technology
from the source of the NKFI.
The work of B.~Vir\'ag was supported by the NSERC Discovery Grant.

\bibliography{bernoulli_exp}

\begin{thebibliography}{FINR04}

\bibitem[Arr81]{A81}
R.~Arratia.
\newblock {Coalescing Brownian motions on $\R$ and the voter model on $\mathbb
  Z$}.
\newblock {\em unpublished partial manuscript}, 1981.

\bibitem[CH23]{CH23}
G.~Cannizzaro and M.~Hairer.
\newblock {The Brownian Castle}.
\newblock {\em Comm. Pure Appl. Math.}, 76(10):2693--2764, 2023.

\bibitem[DNV23]{DNV23}
D.~Dauvergne, M.~Nica, and B.~Vir\'ag.
\newblock Uniform convergence to the {A}iry line ensemble.
\newblock {\em Ann. Inst. Henri Poincaré Probab. Stat.}, 59(4):2220--2256,
  2023.

\bibitem[DOV22]{DOV18}
D.~Dauvergne, J.~Ortmann, and B.~Vir\'ag.
\newblock The directed landscape.
\newblock {\em Acta Math.}, 229(2):201--285, 2022.

\bibitem[DV21]{DV21}
D.~Dauvergne and B.~Vir\'ag.
\newblock The scaling limit of the longest increasing subsequence.
\newblock {\em arxiv:2104.08210}, 2021.

\bibitem[DZ24]{DZ24}
D.~Dauvergne and L.~Zhang.
\newblock Characterization of the directed landscape from the {KPZ} fixed
  point.
\newblock {\em arXiv:2412.13032}, 2024.

\bibitem[FINR04]{FINR04}
L.~R.~G. Fontes, M.~Isopi, C.~M. Newman, and K.~Ravishankar.
\newblock {The Brownian web: Characterization and convergence}.
\newblock {\em Ann. Probab.}, 32(4):2857--2883, 2004.

\bibitem[FN11]{FN11}
P.~J. Forrester and T.~Nagao.
\newblock Determinantal correlations for classical projection processes.
\newblock {\em J. Stat. Mech. Theory Exp.}, 2011(08):P08011, 2011.

\bibitem[For10]{F10}
P.~J. Forrester.
\newblock {\em Log-Gases and Random Matrices (LMS-34)}.
\newblock Princeton University Press, 2010.

\bibitem[Gan21]{G21}
S.~Ganguly.
\newblock Random metric geometries on the plane and {K}ardar--{P}arisi--{Z}hang
  universality.
\newblock {\em Notices Amer. Math. Soc.}, 2021.

\bibitem[Joh01]{Joh01}
K.~Johansson.
\newblock Discrete orthogonal polynomial ensembles and the plancherel measure.
\newblock {\em Ann. of Math.}, 153(1):259--296, 2001.

\bibitem[LPW06]{LPW06}
D.~A. Levin, Y.~Peres, and E.~L. Wilmer.
\newblock {\em {Markov chains and mixing times}}.
\newblock American Mathematical Society, 2006.

\bibitem[NRS10]{NRS10}
C.~M. Newman, K.~Ravishankar, and E.~Schertzer.
\newblock Marking $(1,2)$ points of the {B}rownian web and applications.
\newblock {\em Ann. Inst. Henri Poincaré Probab. Stat.}, 46(2):537--574, 2010.

\bibitem[PS02]{PS02}
M.~Pr{\"a}hofer and H.~Spohn.
\newblock Scale invariance of the {PNG} droplet and the {A}iry process.
\newblock {\em J. Stat. Phys.}, 108:1071--1106, 2002.

\bibitem[QS23]{QS23}
J.~Quastel and S.~Sarkar.
\newblock Convergence of exclusion processes and {KPZ} equation to the {KPZ}
  fixed point.
\newblock {\em J. Amer. Math. Soc.}, 36(1):251–289, 2023.

\bibitem[Sod14]{S15}
S.~Sodin.
\newblock {A Limit Theorem at the Spectral Edge for Corners of Time-Dependent
  Wigner Matrices}.
\newblock {\em Int. Math. Res. Not. IMRN}, 2015(17):7575--7607, 2014.

\bibitem[SS08]{SS08}
R.~Sun and J.~M. Swart.
\newblock {The Brownian net}.
\newblock {\em Ann. Probab.}, 36(3):1153--1208, 2008.

\bibitem[TW98]{TW98}
B.~T\'oth and W.~Werner.
\newblock The true self-repelling motion.
\newblock {\em Probab. Theory Relat. Fields}, 111:375--452, 1998.

\bibitem[Vir20]{V20}
B.~Vir\'ag.
\newblock The heat and the landscape {I}.
\newblock {\em arXiv:2008.07241}, 2020.

\bibitem[Wu23]{W23}
X.~Wu.
\newblock The {KPZ} equation and the directed landscape.
\newblock {\em arxiv:2301.00547}, 2023.

\end{thebibliography}
\bibliographystyle{alpha}
\end{document}